\documentclass[12pt]{amsart}
\usepackage[leqno]{amsmath}
\usepackage{amsfonts}
\usepackage{mathrsfs}
\usepackage{latexsym}
\usepackage{amssymb, euscript}
\usepackage{amscd}
\usepackage{MnSymbol}
\usepackage{amsthm}
\usepackage{mathtools}
\usepackage{enumerate}
\usepackage{tikz-cd}
 \usepackage[english]{babel} 
 \usepackage{graphicx}
\usepackage[all,cmtip]{xy}
\usepackage{dsfont}

\usepackage{stackengine}
\usepackage{scalerel}

\usepackage{mathdots}

\usepackage{tikz}
\usepackage{tikz-cd}

\usepackage{leftidx}

\usepackage{epsfig}  		
\usepackage{epic,eepic}       

\usepackage{tkz-euclide}

\newdimen\headwidth
\newdimen\headrulewidth\label{means}
\headwidth \textwidth

\theoremstyle{plain}
\newtheorem{proposition}{Proposition}
\newtheorem{lemma}[proposition]{Lemma}

\newtheorem{theorem}[proposition]{Theorem}

\theoremstyle{definition}
\newtheorem{definition}[proposition]{Definition}
\newtheorem{example}[proposition]{Example}
\theoremstyle{remark}
\newtheorem{remark}[proposition]{Remark}
\numberwithin{proposition}{section}
\numberwithin{equation}{section}

\newcommand{\ds}{\displaystyle}

\newcommand{\bA}{\mathbb {A}}
\newcommand{\bC}{\mathbb {C}}

\newcommand{\bG}{\mathbb {G}}
\newcommand{\bQ}{\mathbb {Q}}

\newcommand{\bR}{\mathbb {R}}

\newcommand{\bZ}{\mathbb {Z}}

\newcommand{\tw}{\widetilde {w}}

\newcommand{\tM}{\widetilde {M}}

\newcommand{\tv}{\widetilde {v}}
\newcommand{\tT}{\widetilde {T}}

\newcommand{\tpi}{\widetilde {\pi}}
\newcommand{\ts}{\widetilde{\sigma}}
\newcommand{\tPi}{\widetilde {\Pi}}

\newcommand{\der}{\text{der}}
\newcommand{\End}{\text{End}}

\newcommand{\Hom}{\text{Hom}}

\newcommand{\Ind}{\text{Ind}}

\newcommand{\std}{\text{std}}
\newcommand{\Sym}{\text{Sym}}

\newcommand{\Spec}{\text{Spec}}
\newcommand{\DivB}{\text{Div}_B}

\newcommand{\oT}{\overline {T}}

\newcommand{\hG}{\hat{G}}
\newcommand{\hT}{\hat{T}}
\newcommand{\hL}{\hat{L}}

\newcommand{\hphi}{\hat{\phi}}
\newcommand{\hPhi}{\hat{\Phi}}

\newcommand{\sH}{\mathcal{H}}

\newcommand{\sS}{\mathcal{S}}

\newcommand{\sO}{\mathcal{O}}

\usepackage{footnote}

\title[Resolution of Reductive Monoids and Multiplicativity]{On the Resolution of Reductive Monoids and Multiplicativity of $\gamma$-Factors}
\author{Freydoon Shahidi and William Sokurski}

\begin{document}
\date{}
\maketitle

\begin{abstract}
    In this article, we give a proof of multiplicativity for $\gamma$-factors, an equality of parabolically induced and inducing factors, in the context of the Braverman-Kazhdan/Ngo program, under the assumption of commutativity of the corresponding Fourier transforms and a certain generalized Harish-Chandra transform. Within our proof, we define a suitable space of Schwartz functions which we prove to contain the basic function. We also discuss the resolution of singularities and their rationality for reductive monoids, which are among the basic objects in the program.
\end{abstract}

\section*{Introduction}

Every theory of $L$-functions must satisfy the axiom of multiplicativity/inductivity, which simply requires that $\gamma$-factors for induced representations are equal to those of the inducing representations. This axiom is a theorem for Artin $L$-functions and those obtained from the Langlands-Shahidi method \cite{shahidi2010eisenstein}, and is a main tool in computing $\gamma$-factors, root numbers, and $L$-functions. On the other hand, its proof in the cases obtained from Rankin-Selberg methods are quite involved and complicated.  It is also central in proving equality of these factors when they are defined by different methods and in establishing the local Langlands correspondence (LLC) \cite{shahidi2012equality,shahidi2017local, harris2001geometry, henniart2000preuve, gan2011local, 10.1215/00127094-2017-0001} . Its importance as a technical tool in proving certain cases of functoriality \cite{cogdell2004functoriality,kim2003functoriality,kim2000functorial} is now well established.

In this paper we will provide a proof of multiplicativity for $\gamma$-factors defined by the method of Braverman-Kazhdan/Ngo \cite{braverman2002normalized,braverman2010gamma,bouthier2016formal,ngo2020hankel} and L. Lafforgue \cite{lafforgue2014noyaux} in general under the assumption that the $\rho$-Fourier transforms on the group $G$ and the inducing Levi subgroup $L$ commute with the $\rho$-Harish-Chandra transform, a generalized Satake transform sending $C_c^{\infty}(G(k)) \to C_c^{\infty}(L(k))$, where $\rho$ is a finite dimensional representation of the $L$-group of $G$ by means of which the $\gamma$-factors are defined.

Within our proof, we define a space $\sS^{\rho}(G)$ of $\rho$-Schwartz functions for every $\rho$ as 

\begin{equation}\label{def}
 \sS^{\rho}(G):=C_{c}^{\infty}(G(k)) + J^{\rho}(C^{\infty}_c(G(k))) \subset C^{\infty}(G(k)).
\end{equation}
This definition is crucial since the $\rho$-Schwartz functions defined in this way will be uniformly bi-$K$-finite (see equation (5.16) and Lemma $5.5$), making the descent to the inducing level possible, an important step in the proof of multiplicativity. While the $\gamma$-factor can be defined as the kernel of the Fourier transform, it is the full functional equation that allows our descent to the inducing level in a transparent fashion, using our definition of $\rho$-Schwartz functions.

In \cite{braverman2010gamma}, Braverman and Kazhdan defined their Schwartz space as a ''saturation" of ours. But our Schwartz space, which is denoted by $V_{\rho}$ in \cite{braverman2010gamma}, covers a significant part of theirs and in particular, contains the $\rho$-basic function as we prove in Proposition $5.3$ This is done using the extended Satake transform to almost compact functions \cite{li2017basic} and the fact that it commutes with the Fourier transform induced from tori which is now defined in general, cf. Section $6$ and in partiular diagram $(6.8)$.

The commutativity assumption allows us to extend the $\rho$-Harish-Chandra transform to $\sS^{\rho}(G)$, commuting with $J^{\rho}$ and $J^{\rho_L}$, respectively, where $\rho_L$ is the restriction of $\rho$ to the $L$-group of $L$. This construction of $\sS^{\rho}(G)$ agrees with that of Braverman-Kazhdan in the case of doubling method \cite{braverman2002normalized, gelbart2006explicit,li2018zeta,lapid2005local, piatetski1986varepsilon, shahidi2016generalized,jiang2020harmonic,getz2020refined}, since $G$ being the interior of the defining monoid embeds as a unique open orbit into the Braverman-Kazhdan space (cf. \cite{li2018zeta}). Our proof is a generalization of Godement-Jacquet for $GL_n$, Theorem 3.4 of \cite{godement2006zeta}.

One expects $\sS^{\rho}(G) \subset C_c(M^{\rho}(k))$, where the latter is defined as the space of functions of compact support on $M^{\rho}(k)$, the monoid attached to $\rho$ whose restriction to $G(k)$ are smooth (locally constant) and sending $C_c^{\infty}(G(k))$ to itself inside $C_c(M^{\rho}(k))$. The group $G$ being smooth as a variety, the singularities of the monoid $M^{\rho}$ are outside of $G$. Renner's construction (Section 2) realizes $M^{\rho_L} \subset M^{\rho}$ as a closed subvariety for any Levi subgroup $L$ that contains $T$, where $T$ is a fixed maximal torus of $G$, used in the construction of $M^{\rho}$.

The $\rho$-Harish-Chandra transform cannot be defined on $C_c(M^{\rho}(k))$ since the modulus character $\delta_P$ may vanish outside $L(k) \subset M^{\rho_L}(k)$, but it can be defined on the image of $\sS^{\rho}(G)$ inside $C_c(M^{\rho}(k)) \cap C^{\infty}(G(k))$, that lands in the image of $\sS^{\rho_L}(L)$ in $C_c(M^{\rho_L}(k))$ by the above discussion. 

Our commutativity axiom, which implies multiplicativity and multiplicativity itself give rise to an inductive scheme that allows for a definition of Fourier transform $J^{\rho}$ by building from the case of conjugacy classes of Levi subgroups $L$ of $G$. In fact, Theorem 5.3 gives the $\gamma$ factors $\gamma(s,\pi,\rho,\psi)$, $\pi$ an irreducible constituent of $\mathrm{Ind}(\sigma)$, equal to the inducing $\gamma$-factor $\gamma(s,\sigma,\rho_L,\psi)$, which in turn is defined through convolution by $J^{\rho_L}$. For example, for $GL_2$, the Levi subgroups consist of split tori for which a canonical Fourier transform exists (c.f. \cite{ngo2020hankel}; see (6.2) here) and $GL_2$ itself, which is equivalent to understanding supercuspidal $\gamma$-factors. We refer to section 5.4 for a more detailed discussion of this inductive construction. 

In the case of $GL_2$, Laurent Lafforgue \cite{lafforgue2014noyaux} has defined a candidate distrubution which is shown formally to commute with the Harish-Chandra transform and evidence exists that it may give the correct supercuspial factors as observed by Jacquet, but it is still unknown if this is the right distribution. Work in this direction for tamely ramified representations is being pursued by the second author.

Although our definition of the space $\sS^{\rho}(G)$ depends on the knowledge of how $J^{\rho}$ acts on $C_c^{\infty}(G(k))$, this seems to be the most efficient way of defining $\sS^{\rho}(G)$ at present and sufficient for our purposes as a working definition, allowing us to begin making some initial steps toward understanding the general theory, and as observed earlier after equation (0.1), essential in proving the uniform $K$-finiteness of $\rho$-Schwartz functions.

One hopes that the geometry of $M^{\rho}$ will provide some insight into what this Fourier transform ought to be. In fact, the geometric techniques used to study the basic functions on reductive monoids via arc spaces in the function field setting \cite{bouthier2016formal} tells us that the nature of the singularities of the monoid very much controls the asymptotics of the basic function. Taking cue from this, it is natural to consider the geometry of the singularities in the $p$-adic case as well. As a first step, we may classify the singularities of our monoids via the theory of spherical varieties and we find that there is a good and explicit choice of $G$-equivariant resolution of singularities \cite{brion1989spherical,perrinspherical}. The resolution is moreover rational and so we may pass without trouble between differential forms on the monoid and its resolution. The geometric aspects of this theory are discussed in part in Section 3 of the present paper. Since our Schwartz spaces are, at least tentatively, linked by the definition of the Fourier transform $J^{\rho}$ via $\sS^{\rho}(G) = C_c^{\infty}(G(k)) + J^{\rho}(C_c^{\infty}(G(k))$, we are able, at least speculatively, to unite the themes of this paper.   Here is the outline of the paper.

Section 1 is a quick review of the method for $GL(n)$ as developed in \cite{godement2006zeta}. Renner's construction of reductive monoids is briefly discussed in Section 2 which concludes with a treatment of the cases of symmetric powers for $GL(2)$, describing all the objects involved in those cases. Section $3$ covers the geometric aspects studied in the paper. This includes the resolution of the singularities of reductive monoids, leading to a proof of rationality of these singularities. This allows a transfer of measures from the monoid to its resolution as discussed in Section $4$ and can be applied to the integration of basic functions on corresponding toric varieties in Example $4.1$. Multiplicativity is stated and proved in Section $5$, concluding with the example of $GL(n)$ in $5.3$ and a discussion of the inductive nature of Fourier transforms in $5.4$. In proving multiplicativity, we have found it easier to work with the full functional equation rather than the definition given by convolutions. The cases of a tori and unramified data are addressed in Section $6$. The paper is concluded with a brief discussion of the doubling construction of Piatetski-Shapiro and Rallis with relevant references cited. 

\section*{Acknowledgements}

The authors would like to thank J. Getz, D. Jiang, and B.C. Ngo for helpful conversations. A part of this paper was presented by the first author during the month long program ``On the Langlands Program: Endoscopy and Beyond" at the Institute for Mathematical Sciences, National University of Singapore, December 17, 2018-January 18, 2019. He would like to thank the Institute and the organizers: W. Casselman, P.-H. Chaudouard, W.T. Gan, D. Jiang, L. Zhang, and C. Zhu, for their invitation and hospitality. Finally, we would like to thank Jayce Getz, Chun-Hsien Hsu, and Michel Brion for their comments after the paper was posted on arXiv. Both authors were partially supported by NSF grants DMS 1801273 and DMS 2135021

\section{The case of standard representation for $GL_n$} 

We recall that the Godeement--Jacquet \cite{godement2006zeta} theory for standard $L$-functions of
$GL_n$, which this method aims to generalize, can be presented briefly through
the definition of corresponding $\gamma$-factors.

Let $F$ be a $p$-adic field and $G=GL_n$. Let $\pi$ be an irreducible admissible representation of $GL_n(F)$. Given a Schwartz function $\phi$ on $M_n(F)$, i.e.,
 $\phi\in C^\infty_c(M_n(F))$, a smooth function of compact support on $M_n(F)$,
 one can define a zeta-function
 \[Z(\phi,f,s)=\int\limits_{GL_n(F)}\phi(x)f(x) |\det x|^s dx,\]
 where $f(x)=\langle \pi(x) v, \tv\rangle, \  v\in\sH(\pi)$ and
 $\tv\in\sH(\tpi)$, is a matrix coefficient and $s\in\bC$. Here $\tpi$ is the
 contragredient of $\pi$. Let
 \[\hphi(x): =\int\limits_{M_n(F)}\phi(y)\psi(tr(xy))dy \]
 be the Fourier transform of $\phi$ with respect to the (additive) character
 $\psi\ne 1$ of $F$.
 
If $\check{f}(g)=f(g^{-1})$, $g\in GL_n(F)$, then we can consider
$Z(\hphi,\check{f}, s)$. The Godement--Jacquet theory defines a $\gamma$-factor
$\gamma^{\std}(\pi,s)$ which depends only on $\pi$ and $s$ and is a rational
function of $q^{-s}$, satisfying
\begin{equation}   
Z(\hphi,\check{f}, (1-s)+\frac{n-1}{2})=\gamma^{\std}(\pi,s) Z(\phi, f, s+\frac{n-1}{2})
\end{equation}
for all $\phi$ and $f$.

It is not hard to see that if we introduce the Int$(G)$-invariant kernel, $G=GL_n$, 
\[\Phi_\psi(g)=\psi(tr (g)) |\det g|^n dg\]
of the Fourier transform, then
\begin{equation}\label{1.2}
    \Phi_\psi \star f |\det|^{s+\frac{n-1}{2}}=\gamma^\std (\pi,s) f
|\det|^{s+\frac{n-1}{2}}
\end{equation}
by virtu of irreducibility of $\pi$ and the Schur's lemma.

This formulation for the $\gamma$-factor is a quick and convenient way of
introducing them which is amenable to generalization. We can therefore write
\[\gamma^\std(\pi,s)=\Phi_\psi(\pi)=\int\limits_{GL_n(F)}\Phi_\psi(g)\pi(g) dg,\]
pointing to the significance of the kernel $\Phi_\psi$ in defining the
$\gamma$-factors.

\section{The general case; monoids and Renner's construction} 

To treat the general case we need to generalize $M_n(F)$. Let $k$ be an
algebraically closed field of characteristic zero. A monoid $M$ is an affine
algebraic  variety over $k$ with an associative multiplication and an identity
1. For our purposes, we also need $M$ to be normal, i.e., $k[M]$ is integrally 
closed in $k(M)$. We can always find a normalization in case $M$ is not normal,
i.e., an epimorphism $\tM\to M$ such that integral closure of $k[M]$ in $k(M)$
equals $k[\tM]$ as we realize $k[M]\hookrightarrow k[\tM]$. 

We thus let $M$ be a normal monoid and let $G=G(M)=M^*$, be the units of $M$. We
 say $M$ is {\it{reductive}} if $G$ is. We now like to attach a monoid to a
 finite dimensional representation $\rho$ of $\hG={^L}G$, $L$-group of $G$, $\rho:\hG\to GL(V_\rho)$, where $G$ is a split reductive group. Let $T\subset G$ be a maximal torus
 and write 
 \[\rho\ | \ \hT=\mathop{\bigoplus_{\lambda\in W(\rho)}}\lambda,\]
where $W(\rho)$ is the set of weights of $\rho$. Let $\Lambda=\Hom(\bG_m, T)$ be
the set of cocharacters of $T$ or characters of $\hT$ and set
$\Lambda_\bR=\Lambda\otimes{_\bZ}\bR$. 
Next, denote by $\Omega(\rho)$ the convex span in $\Lambda_\bR$ of weights of
$\rho$ and let $\xi(\rho)$ be the cone in $\Lambda_\bR$ generated by rays
through $\Omega(\rho)$.

Let $\sigma^\vee=\xi(\rho)^\vee\cap X^*(T)$, be the ``rational'' dual cone to
$\xi(\rho)\cap X_*(T)$ and $k[\sigma^\vee]$ the group algebra of $\sigma^\vee$. One
can then identify $\sigma^\vee$ as a subset of $k[\sigma^\vee]$ by
$\mu\in\sigma^\vee$
going to $\chi_\mu$, defind by
\[\chi_\mu(\eta)=0\quad\text{unless } \ \eta=\mu, \ \eta\in \sigma^\vee,\]
$\chi_\mu(\mu)=1$, 
and $\chi_{\mu_1}\cdot\chi_{\mu_2}=\chi_{\mu_1+\mu_2}$, where the sum is the one on the
semigroup $\sigma^\vee$. We note that this is valid for any semigroup $S$ and 
\[k[S]=\langle \chi_s | s\in S\rangle.\]

Now, assume $G$ has a character
\[\nu: G\longrightarrow \bG_m\]
such that
\[\bC^*\overset{\nu^\vee}{\longrightarrow} \hG
\overset{\rho}{\longrightarrow}GL(V_\rho)\]
sends $z\in\bC^*$ to $z\cdot\text{Id}$.  
This means that $\langle\nu,\omega\rangle=1$ for any weight $\omega$ of $\rho$. In
fact, for $z\in\bC^*$,
\[z^{\langle\nu,\omega\rangle}=\omega(\nu^\vee(z))=\rho(\nu^\vee(z))=z\]
 and thus
$\langle \nu,\omega\rangle=1$. Then $\nu\in\sigma^\vee$ and its existence
implies that $\xi(\rho)$ is strictly convex, i.e., has no lines in it. In fact,
the cone $\xi(\rho)$ is contained in the open half--space of vectors
$x\in\Lambda_\bR$, $\Lambda_\bR=\Lambda\otimes_\bZ \bR$, $\Lambda=\Hom(\bG_m,
T)$, satisfying $\langle\nu, x\rangle > 0$. It is therefore strictly convex. (cf.
\cite{ngo2020hankel},Proposition 5.1).

By the theory of toric varieties \cite{cox2011toric}, the strictly convex cone
$\xi(\rho)$ determines (uniquely) a {\it{normal}} toric variety, i.e., a
normal affine torus embedding $j:T\subset M_T$. Here $M_T$ is the monoid for $T$
attached to $\rho | \hT$. More precisely, $M_T=\text{Spec}(k[\sigma^\vee])$ by
Theorem 1.3.8, pg. 39 of \cite{cox2011toric}. By definition 3.19 of
\cite{renner2006linear}, $k[\sigma^\vee]$ is generated by $X(M_T)$ and thus
$X(M_T)=\sigma^\vee$, the semigroup defining $M_T$. The embedding $j:T\subset
M_T$, defines $j^*: X(M_T)\hookrightarrow X(T)$, a morphism of semigroups, into the character group of $T$.

The dominant characters in $X(T)$ all lie in $X(M_T)$ and are those that extend to semigroup morphism $M_T\to\bA^1=\bG_a$ (Proposition 3.20 of \cite{renner2006linear}).

Finally we observe that $\nu$ is integral and dominant and thus $\nu\in X(M_T)$.

\bigskip

To proceed, we remark that the Weyl group $W=W(G,T)$ acts on $T, M_T, X(T)$ and $X(M_T)$ in the usual manner.
Thus the dual rational cone $\sigma^\vee$ may be identified with $X(M_T)$, both semigroups, since its group algebra generated by elements of $X(M_T)$ or $\sigma^\vee$, is $k[M_T]$ as we discussed earlier.

Let $\lambda\in X(T)$ be a dominant (and integral) character. $\!\!$Then $\lambda | T_\der$ defines an irreducible finite dimensional (rational) representation $\mu^\circ_\lambda$ of 
$G_\der$, $T_\der=T\cap G_\der$, of highest weight $\lambda | T_\der$. Since
\[\mu^\circ_\lambda | Z(G)\cap G_\der=\lambda | Z(G)\cap G_\der,\]
we can extend $\mu^\circ_\lambda$ to an irreducible rational representation $\mu_\lambda=\mu^\circ_\lambda\otimes (\lambda | Z(G))$ of 
\[G=(G_\der \times Z(G)) / G_\der\cap Z(G).\]

\begin{definition}
$\mu_\lambda$ is called the irreducible representation of $G$ of highest weight $\lambda$, where $\lambda$ is a dominant rational character of $T$.
\end{definition}

This in particular is valid for dominant elements in $X(M_T)$. We note that $\nu\in X(M_T)$ is one such.

\bigskip

Now choose $\{\lambda_i\}^s_{i=1}$ so that $\bigcup^s_{i=1} W\!\cdot\!\lambda_i\subset X(M_T)$ generates $X(M_T)$. Let $(\mu_{\lambda_i}, V_{\lambda_i})$ be the representation attached to $\lambda_i$. Set $\mu=\bigoplus\limits^s_{i=1} \mu_{\lambda_i}$ and $V=\bigoplus\limits^s_{i=1} V_{\lambda_i}$. The character $\nu$ will be among these $\lambda_i$.  We may assume $\lambda_1=\nu$. Define $M_1=\overline{\mu(G)}\subset \End(V)$. We let $M$ be a normalization of $M_1$.

We note that we may take $\lambda_i | T_\der$ to be among the fundamental weights of $G_\der$, with $\mu_{\lambda_1}=\nu$ extending the trivial representation of $G_\der$ since $\nu$ is a representation (character) of $G/G_\der\simeq Z(G) / G_\der \cap Z(G)$.

\bigskip
\subsection{The case of symmetric powers of $GL_2$} 

As an example in this section we consider the symmetric power representations of $GL_2(\bC)$ and describe these objects in this case.

Let $G=GL_2$ and $\rho=\Sym^n : GL_2(\bC)\to GL_{n+1}(\bC)$, the $n$-th symmestric power of the standard representation of $GL_2(\bC)$. Write $\bC^{n+1}=\langle e_1,\dots, e_{n+1}\rangle$ with the basis $e_1,\dots, e_{n+1}$. Let $\{\mu_i\}$ denote the weights of $\Sym^n$. Then we can order them as 
\[\mu_i(\left(
\begin{array}{cc}
x & 0\\
0 & y
\end{array}\right) )=x^iy^{n-i}\quad ((x,y)\in (\bC^*)^2),\]
$i=0,\dots, n$.
We have
\[\xi(\Sym^n)\cap X_*(T)=\bZ_{\ge 0}-\text{span }\{(n-k, k)  | k=0,\dots, n\}\]
inside $\bR^2$ which equals
\[\bZ_{\ge 0}-\text{span }\{(m,l)| m+l\in n\bZ\}.\]
The dual cone to $\{(m,l)| m+l\in n\bZ\}$ is
\[\{(a,b)\in \frac 1n\bZ\times \frac 1n\bZ\  | \  a-b\in \bZ\}.\]
Thus the dual to $\xi(\Sym^n)\cap X_*(T)$ is the $\bZ_{\ge 0}-\text{span}$ of $\{(1,0), (0,1), (\frac 1n,\frac 1n)\}$.
It is a lattice in the shaded area, corresponding to $\sigma^\vee$
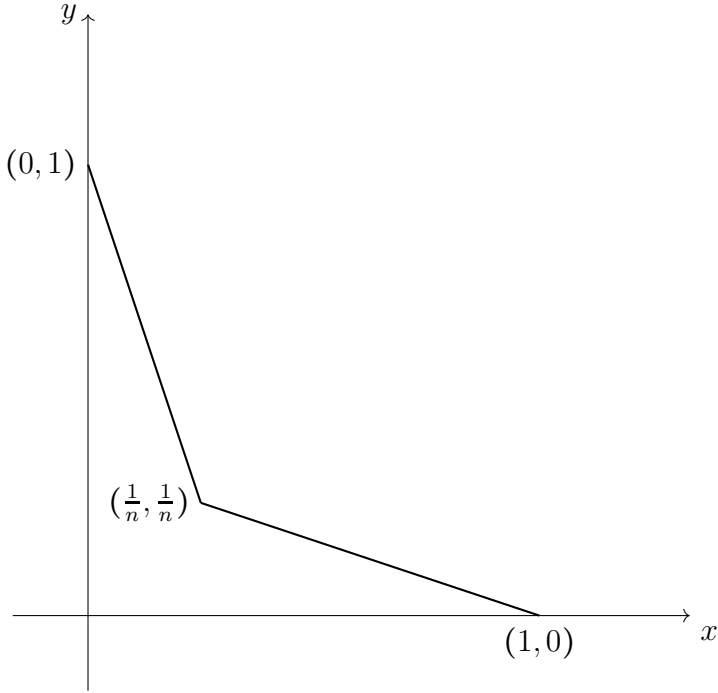
\bigskip


\begin{tikzpicture}[domain=0:2]

\draw[->] (-1,0) -- (8,0)
node[below right] {$x$};
\draw[->] (0,-1) -- (0,8)
node[left] {$y$};
\draw[ thick]  (0,6)node[left] {$(0,1)$} -- (1.5,1.5) node[left]{$(\frac{1}{n},\frac{1}{n})$} ;
    \draw[ thick] (1.5,1.5) -- (6,0)  node[below] {$(1,0)$};
   \tikz\draw[fill=gray!20!white] {(0,6)--(1.5,1.5)--(6,0)--(7.5,0)--(7.5,7.5)--(0,7.5)--(0,6)} ;
  \end{tikzpicture}


\vskip1truein

\noindent
We use $x,y$, and $z$ to denote $(1,0), (0,1)$ and $(\frac 1n, \frac 1n)$ in the semigroup algebra $k[\sigma^\vee]$ as before, i.e., $x=\chi_{(1,0)}$ and so on, then
\[k(x,y,z)=k[X,Y,Z] /  (XY-Z^n).\]
The corresponding toric variety is 
\[M_T=\Spec\,k[X,Y, Z] /  (XY-Z^n)\subseteq k^3,\]
the variety defined by the zeros of $XY-Z^n=0$, and
\begin{eqnarray*}
T&\!\!\!=\hskip-.1truein&M_T\cap (k^*)^3\\
&\!\!\!=\hskip-.1truein&\{(t_1, t^n_2t^{-1}_1, t_2) | t_i\in k^*, \ i=1,2\}
\end{eqnarray*}

\smallskip\noindent
{\underbar{The monoid $M$ for $\Sym^n$}} (Renner's construction): The dual cone in 
\[X^*(T)\otimes_\bZ\bQ=X_*(\hT)\otimes_\bZ\bQ\] is generated by $(1,0)(\frac 1n, \frac 1n)$ and $(0,1)$. The vectors $(1,0)$ and $(0,1)$ are $W$-conjugate and therefore we have as our dominant weights $\lambda_i=\{(1,0),(\frac 1n,\frac 1n)\}$. They correspond, respectively, to std, the standard representation, and $\nu=\det^{1/n}$ (to be explained) and thus
\begin{eqnarray*}
\mu: G\!&\longrightarrow&\!\! \End(V_\std\oplus V_\nu)=M_2\times \bA^1\\
g&\!\longmapsto&\!\! (g, (\det g)^{1/n}).
\end{eqnarray*}
Then
\begin{eqnarray*}
M&=&\!\!\overline{\mu(G)}\\
&=&\!\!\overline{\{(g,a)\ | \ \deg g=a^n}\}\\
&\simeq&\!\!\Spec\,k[X_1,\dots, X_5] \ / \ (X_1X_4-X_2X_3=X^n_5)\\
&=&\!\! Var(X_1X_4-X_2X_3=X^n_5).
\end{eqnarray*}

\medskip\noindent
{\underbar{The character $\nu$ for $\Sym^n$:}} 

Recall that the fibered product of $GL_2$ and $\bG_m$ giving the units of the monoid for $\Sym^n$ is (cf. \cite{shahidi2017local})
\[G=GL_2\times _{\bG_m}\bG_m=\{(g,a) \ | \ \det g=a^n\}=\left\{\begin{array}{cc} GL_2 & n=\text{odd}\\
SL_2\times GL_1 & n=\text{even}
\end{array}\right. .\]
We then have the commuting diagram

\[\begin{CD}
G=GL_2\times  _{\bG_m}\bG_m    @>{\text{Proj}}_1>>   GL_2 \\
\hskip.75truein@V{\text{Proj}}_2VV                               @VV{\det}V \\
\hskip.75truein\bG_m                         @>>>{\bG_m} 
\end{CD}\]
\[ \hskip.75truein x \longmapsto  x^n. \]
Thus
\[\begin{array}{ccc}
(g,a) & {\longmapsto}  &\hskip-.2truein g\\
\downmapsto &   & \hskip-.2truein\downmapsto \\
a & \longmapsto & {\det g=a^n}
\end{array}\]
and the left vertical arrow, the Proj$_2$, gives
\[\nu:(g,a)\longrightarrow (\det g)^{\frac 1n}=a\]
for which
\[z\overset{\nu^\vee}{\longrightarrow}\text{ diag }(z^{1/n},\dots, z^{1/n})\overset{\Sym^n}{\longrightarrow} z\cdot I_{n+1}.
\]

\section{Some geometry of reductive monoids as spherical varieties}

Renner's classification of reductive monoids uses the ``extension principle" \cite{renner2006linear}. The extension principle follows in the spirit of many similar classification results for spherical varieties that rely on the existence of an open $B \times B^{\textrm{op}}$-orbit where $B$ is a Borel subgroup of $G$, that is in Renner's case adapted to account for the monoid structure. By Renner's classification, the category of Reductive monoids is equivalent to the category of tuples $(G,T,\overline{T})$, where $T$ is any maximal torus in $G $ and $\overline{T}$ is a Weyl-group stable toric variety. A morphism of data $(G,T,\overline{T}) \to (G',T',\overline{T'})$ are given by a pair $(\varphi, \tau)$ where $\varphi:G \to G'$ is a morphism of reductive groups and $\tau: \overline{T} \to \overline{T'}$ a morphism of toric varieties such that the restriction of each morphism to the maximal torus agree $\varphi |T = \tau | T$. In the following, we reframe these results in terms of the theory of spherical varieties, in order to state the existence of a $G$-equivariant resolution of singularities \cite{brion2007frobenius}, \cite{rittatore2003reductive}.

Let $G$ be a split reductive group defined over a characteristic zero field $k$. Let $X$ be a variety defined over $k$ with an rational action $\alpha: G \times X \to X$. In this case we say $X$ is a $G$-variety. Let $\mathscr{O}_X$ be the sheaf of regular  on $X$. If $X$ is affine, we will identify $\mathscr{O}_X$ with the coordinate algebra $k[X]$. In this case $\alpha$ induces as usual a co-action map $\alpha^{*}:k[X] \to k[G] \otimes k[X]$ by $(\alpha^{*}f)(x,g)=f(g^{-1}x) = \sum_i h_i(g)f_i(x)$ with $h_i \in k[G]$, where the latter is a finite sum. Thus each $f$ determines a finite dimensional $G$-module. Because $G$ is reductive and we are in characteristic $0$, each finite dimensional $G$-module decomposes as a finite sum of irreducible representations indexed by their highest weight vector with weight $\lambda$. As such we may decompose $k[X] = \bigoplus k[X]_{\lambda}$, indexed by the $\lambda$ that appear in $k[X]$. 

\begin{definition}
A $G$-variety $X$ is \textit{spherical} if $X$ has an open $B$-orbit for some (hence any) Borel $B$ in $G$.
\end{definition}

Suppose $X$ is spherical. Then as above, by highest weight theory, each dominant integral character $\lambda$ of $T(k)$ that appears in $k[X]$ has a highest weight vector $f_{\lambda}$. The line $k\cdot f_{\lambda}$ is the unique line stabilized by $B$ on which $B$ acts through the character $\lambda: f_{\lambda}(bx) = \lambda(b)f_{\lambda}$. In other words $f_{\lambda}$ is a semi-invariant. Suppose $f_1$ and $f_2$ are semi-invariants that are $\lambda$-eigenfunctions appearing in $k[X]$. Then the rational function $f_1 / f_2$ is $B$-invariant. As the $B$-orbit in $G$ is dense, this implies $f_1/f_2$ is constant. Hence for spherical varieties, each $\lambda$ that appears can only appear with multiplicity one. For general reductive groups actions, even the ``naive" (categorical) quotient is reasonably well behaved.

\begin{theorem}
Let $X$ be an affine $G$-variety for a reductive group $G$. Then the ring of $G$-invariants $k[X]^{G}$ is finitely generated, say $k[X]^{G}=k[f_1,\ldots,f_n]$. Then $k[X]^{G} \hookrightarrow k[X]$ defines a good surjective quotient which is moreoever a categorical quotient $q: X \to X//G$. Each fiber of $q$ contains a unique closed $G$-orbit in $X$, and $X//G$ is normal if $X$ is. 
\end{theorem}

\begin{definition}
A spherical variety $X$ is \textit{simple} if it has a unique closed $G$-orbit.
\end{definition}

We are interested in reductive monoids, which have open $G$-orbit and are spherical with respect to $G \times G$ with an open dense borel $= B \times B^{\mathrm{op}}$ orbit. 

\begin{proposition}
Suppose $X$ has an open $G$-orbit. Then $X$ has a unique closed orbit.
\end{proposition}

\begin{proof}
The reductive quotient $q:X \to X//G$ is constant on orbits, in particular on the open orbit. Hence $X//G = \{ \mathrm{pt} \}$. The fiber $q^{-1}(\mathrm{pt})$ contains a unique closed orbit by Theorem 3.2
\end{proof}

Therefore such $X$ are simple. Once again let us consider an affine simple $X$ as a $G \times G$ variety. Decomposing $\bigoplus_{\lambda} k[X]_{\lambda} \cong V_{\lambda}$ where $V_{\lambda}$ is the highest weight module for $\lambda$. The $(B,\lambda)$ eigenfunction $f_{\lambda}$ is $U$-invariant. Thus one may consider taking $U\times U^{\mathrm{op}}$-invariants $k[X]^{U \times U^{\textrm{op}}}$ are therefore generated as a vector space by the $f_{\lambda}$. Using the following

\begin{theorem}
Let $G$ be reductive with maximal unipotent subgroup $U$, and let $X$ be a $G$-variety. Then $k[X]^{U}$ is finitely generated. Moreover $X/U = \operatorname{spec}k[X]^U$ is normal if $X$ is. 

\end{theorem}

\begin{proof}
One first establishes the theorem for $G/U$ i.e. $k[G]^{U}$ is finitely generated and in fact $G/U$ is a geometric quotient (a so called horospherical variety). One has a map $\Phi: X/U \cong X \times ^{G} G/U$ where the quotient is by the diagonal action. On  coordinate rings: a $U$-invariant $f$ defines a $G$-invariant functions $(\Phi^{*}f)(x,gU)=f(gx)$. Thus by Theorem 1.2 $k[X \times G/U]^G \cong k[X]^U$ is finitely generated. 
\end{proof}
We can conclude that the variety $X/(U \times U^{\mathrm{op}})$ is a $T \cong U \backslash BB^{\mathrm{op}} / U^{\mathrm{op}} \hookrightarrow U^{\mathrm{op}} \backslash G / U $ variety, on which $T$ acts on $f_{\lambda}$ through the character $\lambda$. In other words, we have a ring $\bigoplus (V_{\lambda} \otimes V_{\lambda}^{*})^{(U \times U^{\textrm{op}})}$ graded by $\Lambda_X = \{ \lambda \in X^{*}(T): \ k[X]_{\lambda} \neq 0 \}$. By Theorem $3.5$, this is a finitely generated monoid. Each summand has a diagonalizable action by the torus $T$, giving an equivalent characterization of toric varieties: hence defines an affine $T$ embedding. If $X$ is a reductive monoid, this must therefore be equal to Renner's cone. Moreover, if $X$ is normal the associated toric variety is normal, hence the cone of weights of $X$ defining the toric variety is saturated.

\begin{remark}
We have $X$ normal, and the map $X \to X//U$ is faithfully flat. It is sometimes called a \textit{toric degeneration} or \textit{contraction}, see \cite{popov1987contraction}. This is used to show that $X$ has rational singularities if and only if $\overline{T}$ does, because the argument relies on a flat base change argument \cite{elkik1978singularites}.
\end{remark}

Recall that a $G$-variety is simple if it contains a unique closed orbit. When $X$ is affine, it is enough that $G$ embeds as an open subvariety to imply $X$ is simple. We state without proof the following:

\begin{proposition}
Any $G$-variety can be covered by simple $G$-varieties.
\end{proposition}

To classify a general spherical variety one needs the following additional data.

\begin{definition}
Let $V(X)$ denote the $G$-stable discrete valuations on $k[G]$.
\end{definition}

\begin{definition}
Let $\DivB(X)$ denote the set of $B$-stable prime divisors of $X$.
\end{definition}

\begin{definition}
Let $Z$ be a $G$-orbit in $X$. Then $\DivB(X:Z)$ is the set of $B$-stable prime divisors containing $Z$.

\end{definition}

\begin{definition}
Let $\mathcal{B}(X)$ denote the set of irreducible $G$-stable divisors.
\end{definition}

\begin{proposition}
For a divisor $D \in \DivB(X)$, either

\begin{enumerate}
    \item The $B$-orbit $B\cdot D$ is open and dense in the open orbit of $X$.
    \item $D$ is $G$-stable.
\end{enumerate}
\end{proposition}

Briefly (although see \cite{knop1991luna} for details) a simple spherical variety $X$ is determined by its weight monoid, plus the data of which $B$-stable boundary divisors containing the unique closed orbit $Z$ are $G$-stable and which are not. More precisely, each divisor $D$ defines a valuation by first restricting $D \cap G$ which defines a valuation on the multiplicative group of rational functions $k(G)^{\times}$ on $G$ (the valuation $v_{D} \cap G$ is the order of vanishing of a rational function $f/g$ on $D \cap G$). This defines a so-called \textit{colored cone}, defined by the $\bQ_{\geq 0}$ span of the finite number of valuations $v_{D}$ as above in which the colors are the valuations that are $B$-stable but not $G$-stable (or dually, their corresponding divisors have a $B$-open-orbit). 

Thus the set $D(X:Z)$ gives the set of \textit{colors} of the simple spherical variety $X$. The cone generated by $\mathcal{B}(X)$ and the natural image of $D(X:Z)$ in the set of $K[G]$ valuations is the \textit{colored cone} $\mathcal{C}$ determined by the data $(V(X),\mathcal{B}(X))$ that determines up to isomorphism the spherical variety $X$. For reductive monoids, this cone is equivalent to the one constructed in the earlier section via highest weight theory.

\begin{example}
For a reductive monoid $M$, there is a beautiful description of the boundary $\partial M = M \backslash G$ in terms of $B \times B^{\text{op}}$-stable boundary divisors in the form of an extended Bruhat decomposition: Let $R = \overline{N_G(T)} \subset M$ be the Zariski closure of the normalizer of a maximal torus $T$ in $G$. Let $I(M)$ be the set of idempotents in $M$, and note that reductive monoids are \textit{regular} (in Renner's sense), that is, we can decompose $M=G \cdot I(M)$. Then we can construct the \textit{Renner monoid} (sometimes called the Rook monoid) $\mathcal{R} := R/T$. Because reductive monoids are regular, $\mathcal{R}$ makes sense as a finite monoid whose unit group is the Weyl group $W$, and having the property that 

\[
M = \coprod_{x \in \mathcal{R}} BxB.
\]

From this description, it may be deduced that the set $D(M:Z)$, with $Z = \{ 0 \}$ the unique closed orbit in $M$, is given by the codimension one orbits $\overline{Bs_{\alpha}B^{\text{op}}}$ for $s_{\alpha} \in \mathcal{R}$ the simple reflection in the Weyl group determined by the simple root $\alpha$. 
\end{example}

Thus, for a monoid with reductive group $G$ embedded as its unit group, the colors of $M$ are all the $B \times B^{\text{op}}$ stable irreducible divisors of $G$, and thus the monoid is determined purely by the data $\mathcal{B}(M)$ or equivalently $\mathcal{C}(M)$. We state for convenience this form of the classification.

\begin{theorem}
Let $G$ be a reductive group. The irreducible, normal algebraic monoids $M$ with unit group $G$ are the strictly convex polyhedral cones in $X_{*}(T) \otimes \mathbb{Q}$ generated by $D(M)$ and a finite set of elements in $V(G)$.
\end{theorem}
 
 The theory of reductive monoids affords us an explicit description of $\mathcal{B}(M)$ purely in monoid-theoretic terms. 
 
 \begin{definition}
 A spherical variety $X$ is \textit{toroidal} if $D(X)$ is empty.
 \end{definition}
 
 \begin{proposition}
 Suppose the spherical variety $X$ is toroidal and let $\partial X  = X \setminus G$. Let $P_X$ be the $G \times G$ stabilizer of $\partial X$. Then $P$ is a parabolic and moreover satisfies the local structure theorem, i.e. there is a Levi $L \subset P$, depending only on $G$ and a closed $L$-variety $Z$ such that
 
 \[
 P_u \times Z \to X \setminus \partial X
 \]
is an isomorphism. Moreover, $Z$ is a toric variety under L/[L,L].
 \end{proposition}
 
As a consequence of the above isomorphism, the $L$ orbits of $Z$ correspond to $G$ orbits in $X$. Note that when $X=M$ is a reductive monoid, this is precisely Renner's extension theorem \cite{renner2006linear} with $P_M = B \times B^{\text{op}}$ and $Z = \overline{T}$. The proposition implies that the singularities of $X$ are those determined by the cone of the toric variety $Z$.

Let us recall Renner's extension theorem for normal reductive monoids. It states that a morphism of reductive monoids $M \to M'$ is given by the data $(G,T,\overline{T})$ and $(G',T',\overline{T'})$ and a morphism $\varphi: M \to M'$ is equivalent to $\varphi | G \to G'$ and $\tau | \overline{T}: \overline{T} \to \overline{T'}$. Briefly, in $M$ one has an analogue of the open cell which is the image of an open embedding $U^{\mathrm{op}} \times \overline{T} \times U \to U^{\mathrm{op}}\overline{T}U$ which has codimension $\geq 2$ in $M$. Thus one gets an equivariant morphism $(u',t,u) \mapsto \varphi(u')\tau(t)\varphi(u) \in M'$. By normality of $M$, the codimension $\geq 2$ condition extends the map uniquely to $M \to M'$, and one verifies this is in fact a morphism of monoids.

\begin{remark}
The above map in Renner uses an open Bruhat cell analog in the context of monoids, which yields a structure theorem parallel to Proposition $3.16$. More generally, $G$-equivariant dominant morphisms $\varphi: Y \to Y'$ of spherical varieties are in bijection with linear maps $\varphi_*: X_{*}(Y) \otimes \mathbb{Q} \to X_{*}(Y') \otimes \mathbb{Q}$, where $X_{*}(Y)$ and $X_{*}(Y')$ are the lattices of co-weights of the underlying group $G$, such that the image of the colored fan $\mathcal{C}_Y$ is contained within $\mathcal{C}_{Y'}$.

\end{remark}

In view of the above structure theory,  the singularities of a spherical $G$ variety $X$ are determined by those of its associated toric variety $X//U \cong \overline{T}$. Smooth toric varieties are classified as follows.

\begin{theorem}
An affine toric variety $\overline{T}$ is smooth if and only if the extremal rays of the rational polyhedral cone generated by its weight lattice is a basis for for the character lattice $X^{*}(T)$.
\end{theorem}

As a general toric variety is glued from affine toric varities, it is given by a fan consisting of rational polyhedral cones. Thus a toric variety is smooth if and only if its fan consists of rational polyhedral cones whose extremal rays generate (as a $\mathbb{Z}$-module) $X^{*}(T)$. Moreover any peicewise linear morphism of rational polyhedral cones $\mathcal{C} \to \mathcal{C}'$ defines a $T$-equivariant morphism of toric varietie, thus, by Theorem $3.18$, one obtains an algorithm giving a resolution of singularities of a toric variety $\overline{T}$.

\begin{theorem}
Let $\sigma^{\vee}=\mathcal{C}(\overline{T}) \cap X^{*}(T)=$ the monoid of weights of $\overline{T}$. Starting from an extremal ray $\mathcal{C}(\overline{T})$, successively subdividing the cone such that each resulting cone is smooth, defines a smooth fan consisting of smooth cones $\tilde{\sigma}^{\vee}$ such that the inclusion map $\tilde{\sigma}^{\vee} \hookrightarrow \sigma^{\vee}$ defines a canonical $T$-equivariant resolution of singularities $\widetilde{T} \to \overline{T}$.
\end{theorem}

Finally, one may define a $G$-equivariant resolution from Remark $3.17$ and Theorems $3.18$ and $3.19$:

\begin{theorem}
Let $M$ be a reductive monoid. Then there exists a smooth spherical $G$-variety $\tilde{M}$ that is toroidal, and a proper $G$-equivariant morphism $\varphi: \tilde{M} \to M $
\end{theorem}

\begin{proof}

Take the colored cone $\mathcal{C}_M$ determed by $M$. Deleting all colors, construct a fan $\tilde{\mathcal{C}}_M$ generated by a subdivision of $\mathcal{C}_M$ into smooth colorless cones (which always exists, see \cite{cox2011toric}). Then $\tilde{T} \to \overline{T} \subset M$ is a resolution of toric varieties, and $\tilde{M}$ has an affine chart given by the open cell $U^{\textrm{op}} \times \tilde{T} \times U$ by Proposition $3.16$, and therefore is smooth. The natural inclusion of $\tilde{\mathcal{C}}_M$ into $\mathcal{C}_M$ defines a dominant $G$-equivariant morphism $\tilde{M} \to M$, giving us our resolution.
\end{proof}

\begin{example}
The case of $G=GL_2$ and $\rho = \mathrm{Sym}^n$ monoids are determined by their respective toric varieties $\{ xy - z^n=0 \} \cong \mathbb{A}^2/ C_n$, which are realized finite quotients of $\mathbb{A}^2$ by cyclic group $C_n$ of order $n$. These are the well known finite quotient surface singularities of type $A_n$ and their resolution are give by $[\frac{n}{2}]$ blow-ups at the origin. Likewise, the resolution of $M^{\rho}$ is given by $[\frac{n}{2}]$-blow-ups at $0$ of $M^{\rho}$.

\end{example}

Next we discuss the rationality of singularities. Recall that $X$ has rational singularities if a (and hence any) resolution $r:\widetilde{X} \to X$ has vanishing higher cohomology: $R^{i}r_{*}\mathcal{O}_{\widetilde{X}} = 0$ if $i > 0$. Moreover, recall that for a resolution $r: \widetilde{X} \to X$ the fiber over the singular locus $X^{\mathrm{sing}} = r^{-1}(X^{\mathrm{sing}})=E$ is the exceptional divisor. 

\begin{theorem}
Let $X$ be a normal variety with rational singularities. Let $j:X^{\text{sm}} \hookrightarrow X$ denote the embedding of the smooth locus of $X$ into $X$. If we define $\omega_X := j_{*}(\omega_{X^{\text{sm}}})$, by extending (uniquely, by normality) algebraic top-dimensional differentials form to $X$. If $r:\widetilde{X} \to X$ is any resolution of singularities, then $r^{*}\omega_{X}$ extends over $E$ to an algebraic differential form on $\widetilde{X}$.
\end{theorem}

\begin{theorem}
Normal spherical varieties, and hence reductive monoids have rational singularities.
\end{theorem}

\begin{proof}
$X \to X//U \cong \overline{T}$ is a flat deformation onto a toric variety, and toric varieties have rational singularities. Rationality ``descends" under flat deformation (use flat base change, see \cite{elkik1978singularites}), hence $X$ has rational singularities.
\end{proof}

\begin{section}{Integration on singular varieties}
Let $\omega$ be a top-differential form on a reductive monoid $M$. Taking the canonical resolution as above $r:\widetilde{M} \to M$ we get a well defined differential form $r^{*}\omega$ on $\widetilde{M}$. Restricing $r$ to the open set on $\widetilde{M}$ over which $r$ is an isomorphism, we have upon passing to $k$-points

\[
\int_{\widetilde{X}(k)} |r^{*}\omega| = \int_{\widetilde{X}(k)} |\mathrm{Jac}(r)||\omega| = \int_{X(k)} |\omega|,
\]
where $\mathrm{Jac}(r)$ is the Jacobian of $r$ and $| \omega |$ is the measure constructed by Weil from the top-differential form $\omega$ on $X$.

\begin{example}
As a first basic example, we may consider the basic function on toric varieties. Recall that the basic function on a toric variety $\overline{T}$ defined by cone $\sigma$ with dual cone $\sigma^{\vee}$. Suppose the generators of the weight monoid of $\overline{T}$ are $e_1,\ldots,e_n$. The $e_i$ are co-characters $k^{\times} \to T(k)$. For each co-character $\lambda$ in the weight monoid, $\lambda(\varpi\mathscr{O}_k)$ is an open neighbood on which the value of the basic function $f^{\rho}_{\overline{T}}$ is  $\# \{ (a_i) \in \bZ^{n}_{\geq 0} \ | \sum_i a_ie_i = \lambda \}$. It is known \cite{sturmfels1995vector, casselman2017symmetric} that the function 

\[\lambda \mapsto \# \{ (a_i) \in \bZ^{n}_{\geq 0} \ | \sum_i a_ie_i = \lambda \}: \sigma \cap X_{*}(T) \to \bZ_{\geq 0}\]
as an integer valued function of the weight monoid is \textit{quasi-polynomial}, i.e., it is in the sub-algebra of functions on $X_{*}(T)$ generated by polynomials and periodic functions. In particular, this means one may dominate the basic function with a polynomial function on the lattice. 

It is also classical that locally in some coordinate system on $\widetilde{T}(k)$, the map $r: \tT(k) \to \oT(k)$ is given by a monomial transformation, i.e., has the form $(t_1,\ldots,t_n) \mapsto (t_1^{d_{11}}\cdots t_n^{d_{n1}},\ldots,t_1^{d_{n1}}\cdots t_n^{d_{nn}} )$. For example, the $\mathrm{Sym}^2$ toric variety is the cone $\{ xy-z^2=0 \}$ and is resolved affine-locally by the classical ``cylinder" resolution $(x,y,z) \mapsto (xz,yz,z)$. Therefore the $|\mathrm{Jac}(r)|$ is in local coordinates a product of the form $|t_i|^{d_i}$. For the cylinder resolution it is $|z|^2$. Thus as $\varpi^N \to 0$ the Jacobian grows as the inverse of an exponential while the basic function is dominated by polynomial growth. Hence the pullback of the differential form $r^{*}(f^{\rho}\omega)$ makes sense as $\mathrm{val}\circ \mathrm{det} \to \infty$ on $\widetilde{T}(k)$.
\end{example}

\end{section}

\section{Multiplicativity}

\bigskip

Every theory of $L$-functions is expected to satisfy ``multiplicativity'', i.e., the equality of $\gamma$-factors for parabolically inducing and induced data. 
Our goal in the rest of the paper is to establish, under natural assumptions on related Fourier transforms, the multiplicativity in the context of this theory in general. Our proof is a generalization of the standard case of Godement--Jacquet \cite{godement2006zeta}. We first need to connect Renner's construction to parabolic induction.

\bigskip
\subsection{Renner's construction and parabolic induction}   

We start by observing that Renner's construction \cite{renner2006linear} respects parabolic induction. More precisely, let $P\!=\!LN$ be  a parabolic subgroup of $G$, with unipotent radical $N$ and a Levi subgroup $L$ which we fix by assuming $T\subset L$. Now $L$ is a reductive group with a maximal torus $T$ to which Renner's construction applies. Let $\rho_L=\rho |\hL$, where $\hL$ is the connected component of the $L$-group of $L$.

Now $\rho_L | \hT$ gives the same weights as $\rho | \hT$ does and thus $\rho_L$ shares the same toric variety $M_T$ coming from $\rho$. Let $W_L=W(L,T)$. Then each orbit $W\lambda_i$ breaks up to a disjoint union of orbits $W_L\lambda_j$ and thus
\[\bigoplus^s_{i=1} V_{\lambda_i}\ {\big |}_L=\bigoplus^s_{i=1}\bigoplus^{s_i}_{j=1}\ V^L_{\lambda_j},\]
where
\[V_{\lambda_i} \ {\big |} L=\bigoplus^{s_i}_{j=1}\ V^L_{\lambda_j}.\]

\bigskip\noindent
(5.1) {\it{Conclusion: The monoid $M^{\rho_L}$ attached to $\rho_L$ by Renner's construction for $L$, $L=G(M^{\rho_L})$, is the same as the closure of $L$ as a subgroup of $G$ upon action on $V=\bigoplus^s_{i=1} V_{\lambda_i}$,  i.e.}}, $M^{\rho_L}=\overline{\mu(L)}$.

\bigskip
We also need to remark that the character $\nu: G\to \bC_m$ discussed earlier, when restricted to $L$ maybe considered as the corresponding character $\nu_L$ of $L$, i.e., $\nu_L :=\nu | L$. 
In fact, $\nu^\vee$ and $\nu^\vee_L$ both take values in the centers of $\hG$ and $\hL$, respectively. The natural embedding of $\hL\subset \hG$ by means of  the root data of $\hL$ and $\hG$ which are dual to root data of $L$ and $G$ who share the maximal torus $T$, identifies $Z(\hG)$ as a subgroup of $Z(\hL)$. We therefore have the commutative diagram

\bigskip
\[
\xymatrix{
\mathbb{C}^* \ar[r]^{{\nu^\vee}}  \ar[rd]_{\nu_L^\vee} &
\hG  \ar[r]^{\!\!\!\!\!\!\! \!\!\!\rho}&
 GL_N(V\rho)  \\
&\hL\ar[u]^{} \ar[ru]_{\rho_{_L}}}
\]
and consequently $\rho_L\cdot\nu_L^\vee(z)=z\cdot \text{Id}$ as needed.

\bigskip\noindent
{\underbar{\Large{The shift in general.}}}\ To get the precise $\gamma$-factor in general one needs to shift $s$ by $\langle\eta_G,\lambda\rangle$ or $\pi\otimes |\nu|^s$ should change to $\pi\otimes |\nu|^{s+\langle\eta_G,\lambda\rangle}$, with notation as in \cite{ngo2020hankel}, where $\eta_G$ is half the sum of positive roots in a Borel subgroup of $G$ and $\lambda$ the hightest weight of $\rho$. 

In our setting, we need to deal with the representation $\rho_{_L}$ of $\hL$ as well which is not necessarily irreducible. Let $\lambda_1,\dots,\lambda_r$ be the hightest weights of $\rho_{_L}$. We may assume $\lambda_1=\lambda$. The shift will then be $\langle \eta_L, \lambda_1+\cdots+\lambda_r\rangle$. 

\bigskip
Let us define $\delta_{G,\rho}=|\nu|^{\langle 2\eta_G,\lambda\rangle}$, $\delta_{L,\rho_L}=|\nu_{_L}|^{\langle 2\eta_L,\lambda_1+\cdots+\lambda_r\rangle}$ and set
\[\nu_{G/L}=\nu_{G/L},\rho :=\delta_{G,\rho} / \delta_{L, \rho_{_L}}.\]

Finally, let $\delta_P$ be the modulus character of $P=NL$.

\bigskip

\subsection{The $\rho$--Harish--Chandra transform}  

We now recall the $\rho$-Harish--Chandra transform, a generalization of Satake transform. Given $\Phi\in C^\infty_c(G(k))$, define its Harish--Chandra transform $\Phi_P\in C^\infty_c(L(k))$ by
\begin{equation}  
\Phi_P(l)=\delta^{-\frac 12}_P(l)\int\limits_{N(k)}\Phi(nl) dn.
\end{equation}
Next define the $\rho$-Harish--Chandra transform, $\rho$--HC in short, by
\begin{equation}   
\Phi^\rho_P(l)=\nu^{\frac 12}_{G/L,\rho}(l)\ \Phi_P(l).
\end{equation}

\noindent
{\underbar{\Large{Fourier transforms}}}

\bigskip

The conjectural Fourier transform (kernel) $J^\rho$ is supposed to give the $\gamma$-factor $\gamma(s,\pi,\rho)$ for every irreducible admissible  representation $\pi$ of $G(k)$ through the convolution
\begin{equation}   
 J^{\rho} * f |\nu|^{s+\langle\eta_G, \lambda\rangle}=\gamma(s,\pi,\rho) f |\nu|^{s+\langle\eta_G,\lambda\rangle}, 
 \end{equation}
where $\pi$ is an irreducible admissible representation of $G(k)$ and $f(g)=\langle \pi(g)v,\tv\rangle$ is a matrix coefficient of $\pi$. In the context of parabolic induction from $P=NL$ we will also have $\gamma(s,\sigma, \rho_{_L})$ defined by $J^{\rho_{_L}}$. 

\bigskip\noindent
{\underbar{Fourier transforms and Schwartz spaces}}

In this section we define suitable spaces of Schwartz functions on $G(k)$ and $L(k)$, assuming how Fourier transforms $J^{\rho}$ and $J^{\rho_{_L}}$ act on $C^\infty_c(G(k))$ and $C^\infty_c(L(k))$, respectively. Then
\[J^\rho : C^\infty_c(G(k))\longrightarrow C^\infty(G(k))\]
and
\[J^{\rho_{_L}} : C^\infty_c(L(k))\longrightarrow C^\infty(L(k)).\]
Assume $J\strut^{\,\rho}$ and $J\strut^{\,\rho_{_L}}$ commute with the $\rho$--Harish--Chandra transform $\Phi^\rho_P$, i.e., 
\begin{equation}   
 (J^{\rho} \Phi)^{\rho}_P=J^{\rho_{_L}}\Phi^\rho_P 
 \end{equation}
or the following diagram commutes
\begin{equation}   
\begin{CD}
C^\infty_c(G(k)) @>J^\rho>> J^{\rho}(C^\infty_c(G(k))\subset C^\infty(G(k))\\
@V\rho-\rm{HC}VV                 \hskip-1truein@VV{\rho}-\rm{HC}V \\
C^\infty_c(L(k))  @>J^{\rho_L}>> J^{\rho_{_L}}(C^\infty_c(L(k))\subset C^\infty(L(k)). 
\end{CD}
\end{equation}

\bigskip

We now define the Schwartz spaces $\sS^{\rho}(G)=\sS^{\rho}(G(k))$ and $\sS^{\rho_{_L}}(L)=\sS^{\rho_{_L}}(L(k))$ as follows:
\begin{equation}   
\sS^{\rho}(G) : = C^\infty_c(G(k))+J^{\rho}(C^\infty_c(G(k))\subset C^\infty(G(k))
\end{equation}
and
\begin{equation}   
\sS^{\rho_{_L}}(L) : = C^\infty_c(L(k))+J^{\rho_{_L}}(C^\infty_c(L(k))\subset C^\infty(L(k)).
\end{equation}

As we pointed out in the introduction, these are subspaces of the conjectural $\rho$-Schwartz spaces and suitable to our purposes. Moreover, as we prove in Proposition 5.3, they contain the $\rho$ and $\rho_L$-basic functions. We recall that the $\rho$-basic function $\phi^{\rho}$ is the unique one for which

\[
Z(\phi^{\rho},f_s)=L(s,\pi,\rho),
\]
where $f_s$ is the normalized spherical matrix coefficient of $\pi \otimes |\nu|^s$ with $Z$ defined as in $(5.11)$.

Note that $\Phi\mapsto\Phi^{\rho}_P$ sends $C^\infty_c(G(k))$ into $C^\infty_c(L(k))$ and thus (5.5) implies that it also sends
\[J^{\rho}(C^\infty_c(G(k)))\longrightarrow J^{\rho_{_L}}(C^\infty_c(L(k))).\]
 We thus have:

\bigskip

\begin{proposition}   
The $\rho$--Harish--Chandra transform $\Phi^\rho_P$  sends $\sS^\rho(G)$ into $\ds\sS^{\rho}(L)$. In particular, equation $(5.4)$ is valid for our spaces of $\rho$-Swartz functions on $G(k)$ and $L(k)$.
\end{proposition}

We remark that this definition will be needed in our proof of multiplicativity, Theorem $5.3$, the discussion after equation (5.16).

\medskip
\noindent
\begin{remark}
This definition of Schwartz spaces agrees with ideas of Braverman--Kazhdan \cite{braverman2002normalized, getz2020refined} and with the case of standard representation of $GL_n(\bC)$. To wit consider $G=GL_1$, i.e., the Tate's setting, and check it for the $\Phi_0=\text{char}(O_k)$, i.e., the corresponding ``basic function''. Let $\Phi=\text{char}(O^*_k)\in C^\infty_c(k^*)$. Now $J^\rho$ is just the standard Fourier transform
\begin{equation}    
J^{\rho}\Phi(y)=\hPhi(y)=\int\limits_k \Phi(x)\psi(tr(xy)) dx.\end{equation}

It can be easily checked that
\begin{equation*}
\begin{aligned}
\Phi_0&=\rm{char}(P^{-1}_k\setminus O_k) +\hPhi\\ 
&\in C^\infty_c(k^*)+J^\rho(C_c^\infty(k^*)),
\end{aligned}
\end{equation*}
where $P_k$ is the maximal ideal of $O_k$.
\end{remark}
This simple calculation allows us to prove the following general result:

\begin{proposition}
The space $\sS^{\rho}(G)$ contains the $\rho$-basic function.
\end{proposition}

\begin{proof}
Note that when $L=T$ is a maximal torus, the $\rho$-Harish-Chandra transform becomes (a twist of) the Satake transform, and in this case the above diagram $(5.5)$ can be extended to the class of almost compact (ac) spherical functions as defined by Wen-Wei Li in \cite{li2017basic}, and we note that the $\rho$-basic function is amongst this class (see \cite{10.1007/978-3-319-94833-1_11}). The above computation for $\Phi_0$ can be extended to show that the function $f_T^{\textrm{std}}=\rm{char}(\mathbb{A}^n(O_k) \cap T_n(k))$ is also a sum in $C^\infty_c(T_n(k))+J^{\rho_T}(C_c^\infty(T_n(k))$, where $T_n(k) \hookrightarrow \mathbb{A}^n(k)$ is the standard embedding of a maximal torus $T_n \cong \mathbb{G}_m^n $ of $GL_n$ into affine space.

Let $\mathrm{Sat} := \textrm{std}-\rm{HC}$ be this extended Satake transform. Given a decomposition $\phi_{T_n}^{\textrm{std}} = f_1+J^{\textrm{std}}(f_2)$, with $f_1,f_2 \in C^\infty_c(T_n(k))$, the commutativity of (5.5) implies that the standard basic function on $GL_n(k)$, $\phi^{\rm{std}} = \mathrm{Sat}^{-1}(f_1) + \mathrm{Sat}^{-1}( J^{\rm{std}_T}(f_2)) = \mathrm{Sat}^{-1}(f_1) + J^{\rm{std}}(\mathrm{Sat}^{-1}(f_2)),$ lies in $\mathcal{S}^{\rm{std}}(GL_n(k))$ as defined in $(5.6)$.  

Note that here $\textrm{Sat}$ is an isomorphism of $K$-spherical compactly supported functions on $G(k)$ and the Weyl-invariant compactly supported functions on $T(k)/T(\sO_k)$. The basic function on $T_n(k)$, and the functions in its decomposition as $f_1 + J^{\textrm{std}}(f_2)$ are invariant under permutations of the coordinates, and so the above maps are well-defined in the remarks above.

We can deduce the analogous case for a  general $\rho$ from the standard case above as follows: Let $T$ be a maximal torus in $G$ with representation $\rho$ of the dual group of $G$. One obtains a canonical map $\tilde{\rho_T}:T_n(k) \to T(k)$ that extends to a map $\mathbb{A}^n(k) \to M_T(k)$, the target of this map being the toric variety constructed in section $2$ (see Section $6$). The $\rho$-Schwartz space on $T(k)$ can be defined as the image of

\[C^\infty_c(\mathbb{A}_n(k)) \cap C^{\infty}(T_n(k)) \to \rho_{*}(C^\infty_c(\mathbb{A}_n(k)) \cap C^{\infty}(T_n(k))),\] 
the pushforward by $\tilde{\rho_T}$. Then the torus basic function $\phi^{\rho}_T$ can be expressed as $\rho_{*}(\phi^{\rm{std}}_{T_n}).$ Moreover, this pushforward is compatible with the $\rho_T$-Fourier transform on tori, as in diagram $(6.4)$. That is,

\begin{equation*}
\begin{aligned}
\phi^{\rho}_T=\rho_{*}(\phi^{\rm{std}}_{T_n}) &= \rho_{*}(f_1+J_T^{\textrm{std}}(f_2))\\
 &= \rho_{*}(f_1) + \rho_{*}(J_T^{\rm{std}}(f_2))\\
 &= \rho_{*}(f_1) + J^{\rho_T}(\rho_{*}(f_2)),
\end{aligned}
\end{equation*}
which shows that $\phi^{\rho}_T \in C_c^{\infty}(T(k)) + J^{\rho_T}(C^{\infty}_c(T(k))$. Finally, the commutativity of diagram $(6.8)$ allows us to lift this decomposition to a decomposition of the basic function as

\begin{equation*}
\begin{aligned}
\phi^{\rho}&=\mathrm{Sat}^{-1}(\phi^{\rho}_T)\\          &=\mathrm{Sat}^{-1}(\rho_{*}(f_1)) + \mathrm{Sat}^{-1}(J^{\rho_T})(\rho_{*}(f_2)))\\
&= \mathrm{Sat}^{-1}(\rho_{*}(f_1)) + J^{\rho}(\mathrm{Sat}^{-1}(\rho_{*}(f_2)).
\end{aligned}
\end{equation*}
\end{proof}

\bigskip
\noindent
{\it{Multiplicativity.}}\ As we discussed earlier every theory of $L$-functions  must satisfy multiplicativity, an axiom that is a theorem for all the Artin $L$-functions and is the main tool in computing $\gamma$-factors and $L$-functions. 
To explain, let $P=NL$ be a parabolic subgroup of $G$ with a Levi subgroup $L$,  uniquely fixed such that $L\supset T$, the maximal torus of $G$ fixed in our construction throughout. Let $\sigma$ be an irreducible admissible representation of $L(k)$ and let $\rho$ be a finite dimensional complex representation of $\hG$ and $\rho_{_L}=\rho | \hL$ as before. For each irreducible admissible representation $\sigma$ of $L(k)$, we can define the $\gamma$-factors $\gamma(s,\sigma,\rho_{_L})$ and $\gamma(s, \Ind_{P(k)}^{G(k)}\sigma,\rho)$. Multiplicativity states that:
\begin{equation}   
\gamma(s, \Ind_{P(k)}^{G(k)}\sigma,\rho)=\gamma(s,\sigma,\rho_{_L}).
\end{equation}
Here we suppress the dependence of the factors on the non-trivial additive character of $k$.

We note that, since the induced representation $\Ind^{G(k)}_{P(k)}\sigma$ may not be irreducible, the $\gamma$-factor $\gamma(s,\Ind^{G(k)}_{P(k)}\sigma, \rho)$ is defined to be $\gamma(s,\pi,\rho)$, where $\pi$ is any irreducible constituent of $\Ind^{G(k)}_{P(k)}\sigma$. The $\gamma$-factor will not depend on the choice of $\pi$ as the proof below establishes. 

A proof of (5.9) is usually fairly hard for $\gamma$-factors defined by Rankin--Selberg method \cite{jacquet1983rankin,soudry1993rankin}. $\!$On the contrary, (5.9) is a general result within the Langlands--Shahidi \cite{shahidi2010eisenstein} method with a very natural proof.

Our aim here is to give a general proof of (5.9) within the Braverman--Kazhdan/Ngo and Lafforgue programs using (5.4) and Int$(K)$-invariance of $J^\rho$. It follows the arguments given in \cite{godement2006zeta}. We now proceed  to give a proof of (5.9) which we formally state as:

\begin{theorem}    
Let $\sigma$ be an irreducible admissible representation of $L(k)$ and let

\smallskip\noindent
 $\Pi=\Ind\,^{G(k)}_{P(k)} \sigma$. Let $\rho$ be an irreducible finite dimensional  complex representation of $\hG$ and let $\rho_{_L}=\rho | \hL$. Assume the validity of (5.4) and thus (5.5) for Schwartz functions, and that $J\strut^{\,\rho}(k^{-1}x k)=J\strut^{\,\rho}(x)$, $k\in K$, a maximal compact subgroup of $G(k)$ satisfying $G(k)=P(k)K$, which follows from the expected Int$(G)$-invariance of the kernel $J\strut^{\,\rho}$. Then
\[\gamma(s,\Pi,\rho)=\gamma(s,\sigma, \rho_{_L}).\]
\end{theorem}

\begin{proof}
Let $\tPi=\Ind^{G(k)}_{P(k)}\ts$ be the contragredient  of $\Pi$. Choose $v\in\Pi$ and $\tv\in\tPi$. Then a 

\smallskip\noindent
matrix coefficient for $\Pi$ can be written as
\begin{eqnarray}   
f(g)&=&\langle\Pi(g)v,\tv\rangle  \nonumber\\
&=&\int\limits_K\langle v(kg),\tv(k)\rangle_0\  dk,
\end{eqnarray}
where $v(\cdot)$ and $\tv(\cdot)$ are values of the functions in $\Pi$ and $\tPi$, respectively, and $\langle\sigma(l)w,\tw\rangle_0$ is a matrix coefficient for $\sigma$, $w\in\sigma$ and $\tw\in\ts$. Let $\Phi\in \sS^{\rho}(G)$ be a $\rho$-Schwartz function in $C^\infty(G(k))$. We absorb the complex number $s$ in $f$ by replacing $\pi$ by $\pi\otimes |\nu|^s$ and then ignoring it throughout the proof.

Now we have the zeta function
\begin{equation}   
Z(\Phi,f)=\int \Phi(g)f(g)\delta^{1/2}_{G,\rho}(g) dg.
\end{equation}
As explained earlier, the shift $\delta^{1/2}_{G,\rho}$ allows us to get the precise $L$-function at $s$, rather than a shift of $s$, when $\Phi$ is the basic function of $\sigma$ for a spherical representation $\sigma$. 
Using (5.10), (5.11) equals 
\begin{eqnarray}   
Z(\Phi,f)=&\int\limits_{K\times G(k)}\Phi(g) \langle v(kg),\tv(k)\rangle_0\  \delta^{1/2}_{G,\rho}(g) dk\,dg  \nonumber\\
\nonumber \\ 
=&\int\limits_{K\times G(k)} \Phi(k^{-1}g) \langle v(g),\tv(k)\rangle_0\  \delta^{1/2}_{G,\rho}(g) dk\,  dg.
\end{eqnarray}
Write $g=nlh$, $n\in N(k)$, $l\in L(k)$, $h\in K$. Then
\[dg=\delta\strut^{-1}_P(l) dn\, dl\, dh.\]
With notation as in \cite{godement2006zeta}, define:
\begin{equation}   
(h\cdot\Phi\cdot k^{-1})(x): = \Phi(k^{-1}x h).
\end{equation}
Thus 
\begin{equation}   
Z(\Phi,f)=\int\limits_{N(k)\times L(k)\times K\times K} (h\cdot \Phi\cdot k^{-1})(nl)\ \delta^{1/2}_P(l) \langle \sigma(l)v(h),\tv(k)\rangle_0 \  \delta^{1/2}_{G,\rho}(l)\delta\strut^{-1}_P(l) dn\, dl\, dh\, dk. 
\end{equation}

Recall the HC--transform $(h\cdot \Phi\cdot k^{-1})_P$:
\[(h\cdot \Phi\cdot k^{-1})_P(l)=\delta^{-1/2}_P(l) \int\limits_{N(k)}(h\cdot \Phi\cdot k^{-1})(nl) dn.\]
Then (5.14) equals
\begin{eqnarray}   
&\int\limits_{L(k)\times K\times K} \nu_{G/L,\rho}^{1/2} (l) (h\cdot \Phi\cdot k^{-1})_P(l) \langle \sigma(l)v(h),\tv(k)\rangle_0 \   
\delta^{1/2}_{L,\rho_{_L}}(l) dl\, dh\, dk \nonumber \\
&=\int\limits_{L(k)\times K\times K}(h\cdot\Phi\cdot k^{-1})^\rho_P(l)\langle\sigma(l)v(h), \tv(k)\rangle_0\  \delta_{L,\rho_{_L}}^{1/2}(l) dl\,dh\,dk. 
\end{eqnarray}
Since $K$ is compact and $v$ and $\tv$ are smooth functions, there exist matrix coefficients $f^L_i(l)$ of $\sigma$ and continuous functions $\lambda_i$ on $K \times K$ such that 
\[\langle \sigma(l) v(h), \tv(k)\rangle_0 =\sum\limits_i f^L_i(l)\lambda_i(h,k).\]
Similarly, there are Schwartz functions $\Phi_j$ in $\sS^{\rho}(G)$ and continuous symmetric functions $\mu_j$ on $K \times K$ such that
\begin{equation}   
h\cdot\Phi\cdot k^{-1}=\sum\limits_j \Phi_j \mu_j(h,k).
\end{equation}  
This is clearly true if $\Phi\in C^\infty_c(G(k))$, since it will then be uniformly smooth. Otherwise, using (5.16) we have
\[
(h\cdot\Phi\cdot k^{-1})^\wedge=\sum\limits_j \hPhi_j \mu_j(h,k),
\]
where $\hPhi :=J\strut^{\,\rho} \Phi$ for simplicity. But Lemma 5.6, proved later, implies
\begin{eqnarray*}
k\cdot\hPhi\cdot h^{-1}&\!
=\!&\sum\limits_j\hPhi_j \mu_j(k,h)
\end{eqnarray*}
for all $h$ and $k$ in $K$ and thus (5.16) holds for all $\Phi\in S^{\,^\rho}(G)$.
Consequently
\[(h\cdot\Phi\cdot k^{-1})^\rho_P(l)=\sum\limits_j \Phi_{j,P}^\rho(l) \mu_j(h,k)\]
with $\Phi_{j,P}^\rho\in \sS^{\rho_{_L}}(L)$ by Proposition 5.1. Let

\[c_{ij} = \int_{K \times K} \lambda_i \mu_j(h,k) dhdk. \]
Then we have
\begin{equation}   
Z(\Phi,f)=\sum\limits_{i,j} c_{ij}Z(\Phi^\rho_{j,P}, f^L_i) 
\end{equation}

For simplicity of notation, let for each $\Phi\in \sS^{\rho}(G)$, $\hPhi := J^{^\rho}\Phi$. We now calculate 
\[Z(\hPhi,\check f)=\int \hPhi(g)\check f(g)\delta^{1/2}_{G,\rho}(g)\  |\nu(g)|\  dg.\]
One needs to be careful since the involution $g\mapsto g^{-1}$ will now play an important role.
We should also point out that $|\nu(g)|$ needs to be inserted to take into account the appearance of $1-s$ in the left hand side of the functional equation as it is the case in equation (1.1) for $GL_n$. We note that $|\nu(g)|=|\nu(l)|=|\nu_L(l)|$ and that $-s$ will appear as $\tpi\otimes |\nu(\null)|^{-s}$, and thus included in $\tpi$ as $s$ did in $\pi$ as $\pi\otimes |\nu(\null)|^s$. Note that in the case

\smallskip\noindent
 of $GL_n$ and $\rho=\std$, $\nu=\det$, and $\delta_{G,\std}^{1/2}=|\det|^{\frac{n-1}{2}}$ as reflected in (1.1).

\medskip

 We have
\begin{eqnarray}   
Z(\hPhi,\check f)&=&\int\limits_{G(k)} \hPhi(g)f(g^{-1})\delta^{1/2}_{G,\rho}(g)\  |\nu(g)|\   dg  \nonumber \\ 
\\  \nonumber
&=& \int\limits_{G(k)\times K} \hPhi(g)\langle v(k g^{-1}), \tv(k)\rangle_0\ \delta^{1/2}_{G,\rho}(g)\  |\nu(g)|\   dg\ dk.
\end{eqnarray}
Changing $g$ to $gk$, (5.18) equals
\begin{equation}   
 =\int\limits_{G(k)\times K} \hPhi(gk)\langle v(g^{-1}), \tv(k)\rangle_0\  \delta^{1/2}_{G,\rho}(g)\  |\nu(g)|\   dg\ dk.
\end{equation}
Write $g=h^{-1}ln$, $n\in N(k)$, $l\in L(k)$, $h\in K$. Then
\begin{equation}   
dg=d(g^{-1})=\delta_P(l)\  dh\  dl\  dn
\end{equation}
and (5.19) equals
\begin{equation}  
=\int\limits_{K\times L(k)\times N(k)\times K} \hPhi(h^{-1}lnk)\langle v(l^{-1} h), \tv(k)\rangle_0\  \delta^{1/2}_{G,\rho}(l) \delta_P(l)\  |\nu(l)|\   dh \ dl\ dn\ dk.
\end{equation}
\begin{equation*}  
\begin{aligned}
&= \int\limits_{K\times L(k)\times K} (\delta^{1/2}_P(l)\int\limits_{N(k)} (k\cdot \hPhi\cdot h^{-1}) (ln) dn) \langle v(l^{-1} h), \tv(k)\rangle_0\  \delta^{1/2}_{G,\rho} (l) \delta^{1/2}_P(l)\  |\nu_L(l)|\   dh\ dl\ dk   
\\  
 &= \int\limits_{K\times L(k)\times K} (\delta^{-1/2}_P(l)\int\limits_{N(k)} (k\cdot \hPhi\cdot h^{-1}) (nl)dn) \langle v(l^{-1} h), \tv(k)\rangle_0\  \nu^{1/2}_{G/L,\rho} (l)\  \delta^{1/2}_P(l)\ \delta^{1/2}_{L,\rho_L}(l)\  |\nu_L(l)|\  dh\ dl\ dk 
\\  
 &= \int\limits_{K\times L(k)\times K} \nu^{1/2}_{G/L,\rho}(l)(k\cdot\hPhi\cdot h^{-1})_P(l)\langle v(l^{-1}h), \tv(k)\rangle_0\  \delta^{1/2}_{L,\rho_L}(l)\  |\nu_L(l)|\ \delta^{1/2}_P(l)\  dh\ dl\ dk, 
\end{aligned}
\end{equation*} 
which finally equals 
\begin{equation}
 \int\limits_{K\times L(k)\times K} (k\cdot \hPhi\cdot h^{-1})^{\rho}_P(l)\ \delta^{-1/2}_P(l)\langle\sigma(l^{-1})v(h), \tv(k)\rangle_0\  \delta^{1/2}_{L,\rho_L} (l)\  |\nu_L(l)|\ \delta^{1/2}_P(l)\  dh\ dl\ dk. 
\end{equation}

Again, for simplicity for each $\Phi\in \sS^{\rho}(G)$, let $\hPhi$ denote its $\rho$-Fourier transform
\begin{equation}   
\hPhi(x)=(J^\rho * \check\Phi) (x).
\end{equation}

To proceed, we need:
\begin{lemma}        
Let $k$ and $h$ be in $K$. Then 
\[({k\cdot \Phi\cdot h^{-1}})^\wedge=h\cdot\hPhi\cdot k^{-1}.
\]
\end{lemma}

\begin{proof}
We have
\begin{eqnarray}   
({k\cdot \Phi\cdot h^{-1}})^\wedge(x)&=&\int\limits_{M^\rho(k)}(k\cdot\Phi\cdot h^{-1}) (y) J^{\rho}(xy) dy \nonumber \\
\nonumber \\
&=&\int \Phi (h^{-1} y k) J^{\rho}(xy) dy \\
&=&\int \Phi (y) J^{\rho}(x h y k^{-1}) dy, \nonumber
\end{eqnarray}
since $h$ and $k$ are in $K$ whose modulus character is 1. Now using Int$(G)$-invariance of

\smallskip
\noindent
 $J^{^\rho}$, (5.24) equals
\begin{eqnarray*}
&=&\int \Phi(y)J^\rho(k^{-1}x hy) dy\\
&=&\hPhi(k^{-1}x h)\\
&=&(h\cdot \hPhi\cdot k^{-1})(x),
\end{eqnarray*}
completing the proof.
\end{proof}

\begin{remark}
In terms of $J^\rho$ we have proved:
\begin{equation}   
J^{\rho} * (k\cdot\Phi\cdot h^{-1})^\vee= h\cdot(J^{\rho} * \check\Phi)\cdot k^{-1}
\end{equation}
\end{remark}

We now apply Lemma 5.5 to equation (5.22) to get:
\begin{equation}    
Z(\hPhi,\check f)=\int\limits_{L(k)\times K\times K} {(({h\cdot \Phi\cdot k^{-1}})^\wedge})\strut^{\rho}_P(l)\langle \sigma(l^{-1})v(h), \tv(k)\rangle_0 \ \delta^{1/2}_{L,\rho_{_L}}(l) \ |\nu_L(l)|\  dl\, dh\, dk.
\end{equation}
But
\begin{eqnarray*}
({h\cdot \Phi\cdot k^{-1}})^\wedge&=&(\sum\limits_j\Phi_j )^\wedge \cdot \mu_j(h,k)\\
\\
&=&\sum\limits_j\hPhi_j \cdot \mu_j(h,k)
\end{eqnarray*}
and
\begin{eqnarray*}
\langle \sigma(l^{-1})v(h), \tv(k)\rangle_0&=&\sum\limits_i f^L_i(l^{-1}) \cdot \lambda_i(h,k)\\
\\
&=&\sum\limits_i \check f^L_i(l) \cdot \lambda_i(h,k)
\end{eqnarray*}
and therefore (5.26) equals
\begin{equation}            
\sum\limits_{i,j}c_{ij}\int\limits_{L(k)}(J^{\rho}\Phi_j)^{\rho}_P(l) \check f^L_i(l) \delta^{1/2}_{L,\rho_{_L}}(l) \ |\nu_L(l)|\  dl.
\end{equation}

We can now apply the commutativity of $\rho$--Harish--Chandra transform  and Fourier transforms $J^{\rho}$ and $J^{\rho_{_L}}$, i.e., equation (5.4) to conclude that (5.27) equals 
\begin{equation}    
\sum\limits_{i,j} c_{ij} \int\limits_{L(k)}J^{{^\rho_{_L}}}(\Phi^\rho_{j,P})(l)\check f_i^L(l)\delta^{1/2}_{L,\rho_{_L}}(l) \ |\nu_L(l)|\  dl.
\end{equation}
But
\[Z(\hPhi,\check f)=\gamma(\Pi,\rho) Z(\Phi,f)\]
by the functional equation for $G$. On the other hand the functional equation for $L$ gives (5.28) as 
\[\gamma(\sigma,\rho_{_L})\sum\limits_{i,j} c_{ij} \int\limits_{L(k)} \Phi_{j,P}^\rho (l) f^L_i (l) \delta^{1/2}_{L,\rho_{_L}}(l) dl\]
which equals 
\begin{eqnarray*} 
&=&\!\!\!\gamma(\sigma, \rho_{_L})\sum\limits_{i,j} c_{ij} Z(\Phi_{j,P}^\rho, f^{L}_i)\\
\\
&=&\!\!\!\gamma(\sigma,\rho_{_L}) Z(\Phi,f)
\end{eqnarray*}
by (5.17). The equality
\[\gamma(\Ind^{G(k)}_{P(k)}\sigma,\rho)=\gamma(\sigma,\rho_{_L})\]
is now immediate.
\end{proof}

\bigskip\noindent
\subsection{The case of $GL_n$.} We now determine $\nu_{_{G/L}}$ in the case $G=GL_n$ and $\rho=\std$, i.e., that of Godement--Jacquet \cite{godement2006zeta} and show that it agrees with calculations in Lemma 3.4.0 of loc. cit., after a suitable normalization. We thus assume $P=NL$ is the standard maximal parabolic subgroup of $GL_n$, containing the subgroup of upper triangular elements $B$, $N\subset B$, with $L=GL_{n'} \times GL_{n^{''}}, n=n'+n^{''}$. Recall that we need to determine $\nu_{_{G/L},\std}=\delta_{G,\std}\, /\, \delta_{L,\std}$. But for $g=\text{diag}(g', g^{''})\in L$

\begin{eqnarray*}
|\det{g'}|^{-n^{''}}\cdot\, \delta_G^{1/2}(g',g^{''})&=& |\det{g'}|^{-n^{''}}\, |\det g|^{\frac{1}{2}(n'+n^{''}-1)}\\
&=&|\det g'|^{\frac{1}{2}(n'-1)}\, |\det g^{''}|^{\frac{1}{2}(n^{''}-1)}\cdot |\det g'|^{-n^{''}/2}\  |\det g^{''}|^{n'/2}\\
&=& \delta_L^{1/2}(g',g^{''})\cdot\, |\det g'|^{-n^{''}/2}\  |\det g^{''}|^{n'/2}.
\end{eqnarray*}
Thus
\begin{equation}    
\nu^{1/2}_{_{G/L}} (g', g^{''})=|\det g'|^{n^{''}\!/2}\,\, |\det g^{''}|^{n'/2}.
\end{equation}
Moreover
\begin{equation}    
\delta_P(g',g^{''})=|\det g'|^{n^{''}}\cdot\, |\det g^{''}|^{-n'}
\end{equation}       
and thus
\begin{equation}    
\nu^{1/2}_{_{G/L}}\, \delta^{-1/2}_P(g', g^{''})=|\det g^{''}|^{n'}.
\end{equation}

We now verify that Lemma 3.4.0 of \cite{godement2006zeta} is equivalent to our commutative diagram (5.5).

\bigskip
Let $J=J^{\std}$ be the standard Fourier transform for $C^\infty_c(M_n(k))$ and $J_L$ its restriction to $C^\infty_c(M_{n'}(k)\times M_{n^{''}}(k))$.  With notation as in pages 37--38 of \cite{godement2006zeta},
\begin{equation}    
\varphi_{{\Phi}}(x,y)=\int_{k^{n'n^{''}}} \Phi (\left(
\begin{array}{cc}
x & u\\
0 & y
\end{array}\right))\  du,
\end{equation}
where $u\in M_{n'\times n^{''}}(k)$ and $\Phi\in C^\infty_c(M_n(k))$, is the analogue of our HC--transform. In fact, (5.32) can be written as
\begin{eqnarray}   
\varphi_{\Phi}(x,y)&=&\int \Phi (\left(
\begin{array}{cc}
I & uy^{-1}\\
0 & I
\end{array}\right) 
\left(
\begin{array}{cc}
x & 0\\
0 & y
\end{array}\right))\  du \nonumber \\
&=& |\det y|^{n'} \int\limits_{N(k)} \Phi(n l) dn,
\end{eqnarray} 

\bigskip
\[=\nu^{1/2}_{G/L}(l)\ \delta^{-1/2}_P(l)\int\limits_{N(k)}\Phi(nl)\ dn\]
by (5.31), where $N=M_{n'\times n^{''}}$ and $l=\text{diag}(x,y)$.

In the notation of \cite{godement2006zeta}, Lemma 3.4.0 of \cite{godement2006zeta} states that
\begin{equation}   
\varphi_{\hPhi}(x,y)=\widehat{\varphi_{\Phi}}(x,y),
\end{equation}
for $\Phi\in C^\infty_c(M_n(F))$ with $\hPhi$ its standard Fourier transform.

Then by (5.33) the right hand side of (5.34) equals
\begin{eqnarray}   
\widehat{\varphi_{\Phi}}(x,y) &=& J_L(\int\Phi (\left( 
\begin{array}{cc}
x & u\\
0 & y
\end{array}\right))\  du) \nonumber \\
&=& J_L(|\det y|^{n'}\int\limits_{N(k)}\Phi(nl)\  dn)  \\
&=& J_L(\nu^{1/2}_{G/L}(l)\Phi_P(l))  \nonumber \\
\nonumber \\
&=& J_L(\Phi^{\std}_P(l)) \nonumber
\end{eqnarray}
by (5.31), where
\[\Phi_P(l)=\delta^{-1/2}_P(l)\int\limits_{N(k)}\Phi(nl)\ dn\]
as defined by (5.1).

Similarly from the left hand side of (5.34), using a change of variables as in (5.33), we have 
\begin{eqnarray}   
\varphi_{\hPhi}(x,y) &=& \int \hPhi (\left( 
\begin{array}{cc}
x & u\\
0 & y
\end{array}\right))\  du \nonumber \\
&=& |\det y|^{n'}\int\limits_{N(k)}\hPhi(nl)\  dn  \\
&=& \nu^{1/2}_{G/L}(l)\ (\hPhi)_P(l)  \nonumber \\
\nonumber \\
&=& (J\Phi)^{\std}_P (l). \nonumber
\end{eqnarray}

Thus (5.34) is equivalent to (5.4) for $GL_n$ and $\rho=\std$.

\subsection{Inductive definition of $J^{\rho}$}

In the introduction we mentioned that multiplicativity plus a definition of Fourier transform that acts through the correct scalar factors equal to the gamma factors on supercuspidal representations/characters, is enough to characterize the full Fourier transform. Indeed, if we assume that $J^{\rho}$ is a good distribution in the sense of Braverman-Kazhan \cite{braverman2010gamma}, then we can identify $J^{\rho}$ with a rational, scalar valued function $\pi \mapsto \gamma(\rho,\pi)$, where $\gamma(\rho, \pi)$ is defined by $J^{\rho} \star \pi = \gamma(\rho, \pi) \pi$.

Our results on mulitplicativity allow in principle for us to construct in an inductive fashion a distribtion $J^{\rho}$ on $G$ by formally inducing from $J^{\rho_L}$ for each conjugacy class of Levi subgroup $L \subset G$. In fact, our setup and definitions, culminating in Theorem $5.3$, are normalized so as to make induction of representations adjoint to our $\rho$-Harish-Chandra transform, that is, we have an equality

\[
\langle J^{\rho}, \Ind_L(\theta) \rangle = \langle J^{\rho_L}, \theta \rangle.
\]

Here $\theta$ is a supercuspidal character of a representation on $L$. The ajdunction allows us to identify the $J^{\rho}$ and $J^{\rho_L}$ actions on the Bernstein components of $\Ind_L(\sigma) = \pi$ on $G(k)$ and the Bernstein component of $\sigma$ on $L(k)$, respectively. In $5.3$, we started with an assumption of knowledge of $J^{\rho}$ and $J^{\rho_L}$ and we showed that this is equivalent to an equality of gamma factors. However, the gamma factors determine the distribution uniquely, and so one can in principle characterize completely a distribution $J^{\rho}$ by specifying its action on supercuspidal representations on $G(k)$, and postulating multiplicativity as an axiom. More concretely, if we inductively know $J^{\rho_L}$ for conjugacy classes of parabolic subgroups $L$, we may formally induce to provide a definition of $J^{\rho}$ with a correct action, at least on functions whose spectral decomposition consists solely of induced data from $L$:
\[ \langle \Ind_L(J^{\rho_L}), f \rangle = \langle J^{\rho}, HC(f) \rangle. \]
The distribution $\Ind_L(J^{\rho_L})$ can be a priori defined by the above in order to meet this adjunction, and in fact $J^{\rho_L}$ will then be represented by the conjugation-invariant function 
\[\Ind_L(J^{\rho_L}) : x \mapsto |D_G(x)|^{\frac{-1}{2}} \sum_y |D_L(y)|^\frac{1}{2} J^{\rho_L}(y)\]
where the $y$ are chosen representatives of $L(k)$-conjugacy classes of elements that are $G(k)$-conjugate to $x$, and $D_G$  and $D_L$ are the respective discriminant functions on $G$ and $L$. (Here we are identifying $J^{\rho_L}$ with the invariant function representing it).

That $\Ind_L(J^{\rho_L})$ satisfies the first adjunction, and therefore multiplicativity, follows from the formula for the trace,  and the expression of the distribution character $\Theta_{\pi} = \Ind_L(\Theta_{\sigma})$ in terms of $\Theta_{\sigma}$, adapted to the $\rho$-setting.

\begin{section}{Example: The case of Tori and unramified data}

We now consider the case of tori, which for present purposes we assume are split. Let $T$ be a split torus over $k$. When $T$ is a maximal split torus in a reductive group $G$, the upcoming discussion gives the first term of the inductive construction defining the Fourier transform for $L=T$, with minimal parabolic $P_0=P=LN=TN$ which is a Borel subgroup. Let $\rho = \rho_T$ be a finite dimensional represenation of $\hT$. Our notation is justified if we assume $\rho_T = \rho|\hT$, where $\rho$ is a representation of $\hG$. Let $n=\text{dim }\rho_T$. Then

\[
\rho_T: \hT \to GL_n(\bC).
\]
Write

\begin{equation} 
\rho_T = \mu_1 \oplus \cdots \mu_n,
\end{equation}
where the $\mu_i, 1 \leq i \leq n$, are the weights of $\rho_T$. We note that they are not necessarily distinct. If we realize these weights of $\hT$ as co-characters of $T$, we get a map $\tilde{\rho_T}: \bG_m^n \to T$ (defined over $k$, as $T$ is split), which being dual to $\rho_T$, is given by (c.f. \cite{ngo2020hankel})

\[
\tilde{\rho_T}(x_1,\ldots,x_n) = \mu_1(x_1)\cdots \mu_n(x_n).
\]
We can extend this to a monoid homomorphism

\[\tilde{\rho_T}:\bA^n \to M^{\rho_T}, \]
where $M^{\rho_T}$ is the corresponding toric variety. As in \cite{ngo2020hankel}, define the trace function $h: \bA^n \to \bA$ by 

\[ h((x_i)) = \sum_i x_i \]
and set 

\[h_{\psi}:k^n \to \bC^{*} \]
by $x \mapsto \psi(h(x))$, where $\psi$ is our fixed non-trivial character of $k$. 

\bigskip
Denote by $J^{\text{std}}$ the kernel

\[ J^{\text{std}}(g) = \psi(\text{tr}(g))|\text{det}g|^ndg   \]
for $g \in GL_n(k)$ as defined in Section 1, i.e., the standard Fourier transform on $M_n(k)$. We use again $J^{\text{std}}$ for its restrition to $\bA^n$, the monoid for $T_n=\bG_m^n$.

\bigskip
In \cite{ngo2020hankel} Ngo defines the kernel $J^{\rho_T}$ for the Fourier transform on $T$ by 

\begin{equation} 
J^{\rho_T}(t) = \int_{\widetilde{\rho_T}^{-1}(t)} h_{\psi}(x)dx,
\end{equation}
which equals to 

\begin{equation} 
J^{\rho_T}(t) = \int_{x \in U(k)} h_{\psi}(xt)dx
\end{equation}
where $U$ is the kernel of $\tilde{\rho_T}$. In Proposition 6 of \cite{ngo2020hankel}, Ngo regularizes this integration into a principal value integral.

The space of Schwartz functions on $k^n$ are compactly supported functions in $k^n$ that are restrictions of standard Schwartz functions on $M_n(k)$ to $k^n$. Their further restriction to $T_n(k)$ is $\sS^{\text{std}}(T_n)$ in our notation.

\bigskip
Let $\rho_*$ be the push-forward of $\tilde{\rho_T}$. We will verify that the diagram 

\begin{equation}   
\begin{CD}
\sS^{\text{std}}(T_n) @>\rho_*>> \sS^{\rho_T}(T)\\
@VJ^{\text{std}}VV                 @VVJ^{\rho_T}V \\
\sS^{\text{std}}(T_n)  @>{\rho_*}>> \sS^{\rho_T}(T). 
\end{CD}
\end{equation}
commutes, where $\sS^{\rho_T}(T)$ is the image of $\sS^{\text{std}}(T_n)$ under $\rho_*$.

\bigskip
Let $\phi \in \sS^{\text{std}}(T_n)$ and define

\begin{equation} 
\tilde{\phi}(\tilde{t}) = \int_{U(k)} \phi(u\tilde{t})du
\end{equation}
where $\tilde{t} \in T(k)$. The commutativity of (6.4) is equivalent to 
\begin{lemma}
For $\phi \in \sS^{\text{std}}(T_n)$, define $\tilde{\phi}$ by (6.5). Then 

\[
\rho_{*}(J^{\text{std}} \star \phi^{\vee}) = J^{\rho_T} \star \tilde{\phi}^{\vee}
\]
\end{lemma}

\begin{proof}
By definition, for $t \in T$,

\begin{equation}
\begin{aligned}
\rho_*(J^{\text{std}} \star \phi^{\vee})(t) &= \int_{U(k)} (J^{\text{std}}\star \phi^{\vee})(ut)du\\
&= \int_{U(k)} ( \int_{T_n(k)} h_{\psi}(ut\tilde{t})\phi(\tilde{t})d\tilde{t})du\\
&= \int_{T_n(k)} h_{\psi}(t \tilde{t}) ( \int_{U(k)} \phi(u^{-1} \tilde{t})du)d\tilde{t}\\
&= \int_{T_n(k)} h_{\psi}(t \tilde{t}) \tilde{\phi}(\tilde{t})d\tilde{t}\\
&= \int_{T(k)} ( \int_{U(k)} h_{\psi}(ut\tilde{t}) du) \tilde{\phi}(u \tilde{t}) d\tilde{t}\\
&= \int_{T(k)} ( \int_{U(k)} h_{\psi}(u \tilde{t})du) \tilde{\phi}(t^{-1} \tilde{t}) d\tilde{t} \\
&= (J^{\rho_T} \star \tilde{\phi}^{\vee})(t),\\
\end{aligned}
\end{equation}
using $T=T_n/U$ in (6.6), then the lemma follows.
\end{proof}

The push-forward $\rho_*$ can be restricted to 

\[
\bC[\hT_n]^{W_n} \cong \sH(T_n(k),T_n(\sO_k))^{W_n}
\]
leading to

\begin{equation}   
\begin{CD}
\sH(T_n(k),T_n(\sO_k))^{W_n} @>\rho_*>> \sH(T(k),T(\sO_k))^{W}\\
@VJ^{\text{std}}VV                 @VVJ^{\rho_T}V \\
\sH(T_n(k),T_n(\sO))^{W_n}  @>{\rho_*}>> \sH(T(k),T(\sO_k))^{W}, 
\end{CD}
\end{equation}
\end{section} 
where $W_n$ is the Weyl group $W_n=W(Gl_n,T_n)$ ,$W:=W(G,T)$ and $\sH$ denotes the corresponding Hecke algebra. Identifying, via the corresponding Satake isomorphisms

\[
\sH^{\circ}(GL_n(k)) := \sH(GL_n(k),GL_n(\sO_k)) \cong \sH(T_n(k),T_n(\sO_k))^{W_n}
\]
and
\[
\sH(G(k),G(\sO_k)) \cong \sH(T(k),T(\sO_k))^W
\]

\begin{equation}   
\begin{CD}
\sH^{\circ}(GL_n(k))  @>Sat>> \sH(T_n(k),T_n(\sO_k))^{W_n} @>\rho_*>> \sH(T(k),T(\sO_k))^{W} @>Sat^{-1}>> \sH(G(k),G(\sO_k)) \\
@VJ^{\text{std}}VV                 @VVJ^{std}V  @VVJ^{\rho_T}V        @VJ^{\rho}VV\\
\sH^{\circ}(GL_n(k))  @>{Sat}>> \sH(T_n(k),T_n(\sO_k))^{W_n} @>\rho_*>> \sH(T(k),T(\sO_k))^{W} @>Sat^{-1} >> \sH(G(k),G(\sO_k)),
\end{CD}
\end{equation}
in which $J^{\rho_T}$ defines the Fourier transform

\[
J^{\rho}: \sS(G) \to \sS(G)
\]
restricted to $\sH(G(k),G(\sO_k))$. Consequently, at least on $\sH(G(k),G(\sO_k))$, the Fourier transform $J^{\rho}$ and $J^{\rho_T}$ commute with the Harish-Chandra transform. 

\section{The case of standard $L$-functions for classical groups; the doubling method}

We conclude by addressing multiplicativity in the case of standard $L$-functions, twisted by a character, for classical groups as developed by Piatetski-Shapiro and Rallis, which has been addressed further by a number of other authors \cite{braverman2002normalized,jiang2020harmonic,li2018zeta,shahidi2016generalized} within our present context. We refer to the local theory developed by Lapid and Rallis. We will be brief and only mention the relevant statements. 

The $\rho$-Harish-Chandra transform is the one given in Proposition $1$ of \cite{lapid2005local} as $\Psi(\omega,s)$ with notation as in \cite{lapid2005local}. Our commutativity equation ($5.4$) in this case is equation ($17$) in Lemma $9$ of \cite{lapid2005local} in which $J^{\rho}=M^{*}_{\nu}(\omega,A,s)$, a normalized intertwining operator, while $J^{\rho_L}$ acts as the operator induced from $M^{*}_{\mathcal{W}}(\omega,B,s)$ with notation as in \cite{lapid2005local} in the context of doubling construction, or simply put $J^{\rho_L} = M^{*}_{\mathcal
{W}}(\omega,B,s)$.

\bibliographystyle{alpha-abbrvsort}
	\bibliography{ref}
	\nocite{*}

	\bigskip 
\noindent Purdue University\\
West Lafayette, IN 47907\\
shahidi@math.purdue.edu\\
wsokursk@purdue.edu
\end{document}